\documentclass[10pt,a4paper]{article}
\usepackage{geometry}
\usepackage{todonotes}
\usepackage{amsmath}
\usepackage{amsthm}
\usepackage{amssymb}
\usepackage{algorithm}
\usepackage{algpseudocode}
\usepackage{float}
\usepackage{tikz}
\usetikzlibrary{automata}
\usetikzlibrary{patterns}
\usetikzlibrary{arrows,shapes}
\usepackage{diagbox}

\usepackage{graphicx}

\usepackage{amssymb,amsmath,verbatim,tabularx,enumerate}

\newcommand{\N}{\mathbb{N}}

\newcommand{\R}{\mathbb{R}}

\newcommand\1{{\mathbf 1}}
\newcommand{\eps}{\varepsilon}

\newcommand{\x}[1]{x^{(#1)}}

\newcommand{\binvar}[1]{\{ 0,1 \}^{#1}}
\newcommand{\conv}[1]{\convexh\left( #1\right)}

\DeclareMathOperator*{\argmax}{arg\,max} 
\DeclareMathOperator*{\convexh}{conv}

\DeclareMathOperator*{\opt}{opt}
\DeclareMathOperator*{\app}{approx}
\DeclareMathOperator*{\adapt}{adapt}

\newtheorem{theorem}{Theorem}
\newtheorem{lemma}[theorem]{Lemma}
\newtheorem{remark}[theorem]{Remark}
\newtheorem{corollary}[theorem]{Corollary}
\newtheorem{proposition}[theorem]{Proposition}

\newtheorem{example}[theorem]{Example}

\usepackage{multirow}
\usepackage{authblk}

\providecommand{\keywords}[1]{\textit{#1}}

\begin{document}
\title{Approximation Algorithms for Min-max-min Robust Optimization and K-Adaptability under Objective Uncertainty}
\author[1]{Jannis Kurtz\footnote{j.kurtz@uva.nl}}
\affil[1]{Amsterdam Business School, University of Amsterdam, 1018 TV Amsterdam, Netherlands }

\date{}

\maketitle

\begin{abstract}
In this work we investigate the min-max-min robust optimization problem and the $k$-adaptability robust optimization problem for binary problems with uncertain costs. The idea of the first approach is to calculate a set of $k$ feasible solutions which are worst-case optimal if in each possible scenario the best of the $k$ solutions is implemented. It is known that the min-max-min robust problem can be solved efficiently if $k$ is at least the dimension of the problem, while it is theoretically and computationally hard if $k$ is small. However, nothing is known about the intermediate case, i.e. $k$ lies between one and the dimension of the problem. We approach this open question and present an approximation algorithm which achieves good problem-specific approximation guarantees for the cases where $k$ is close to or where $k$ is a fraction of the dimension. The derived bounds can be used to show that the min-max-min robust problem is solvable in oracle-polynomial time under certain conditions even if $k$ is smaller than the dimension. We extend the previous results to the robust $k$-adaptability problem. As a consequence we can provide bounds on the number of necessary second-stage policies to approximate the exact two-stage robust problem. We derive an approximation algorithm for the $k$-adaptability problem which has similar guarantees as for the min-max-min problem. Finally, we test both algorithms on knapsack and shortest path problems and related two-stage variants. The experiments show that both algorithms calculate solutions with relatively small optimality gap in seconds.

\vspace{0.1cm}
\keywords{Robust Optimization \and Min-max-min \and K-adaptability \and Approximation \and Algorithm}
\end{abstract}

\section{Introduction}
Integer optimization problems nowadays emerge in many industries as production, health care, disaster management or transportation, just to name a few. The latter problems are tackled by companies, non-profit organizations or governmental institutions and providing good solutions is a highly relevant topic in our society. Usually solving optimization problems in practice involves uncertainties which have to be incorporated into the optimization model. Typical examples are uncertain traffic situations, demands or failures of a network. The optimization literature provides several ways to model uncertainties, e.g. stochastic optimization \cite{birge2011introduction}, robust optimization \cite{ben2009robust} or distributionally robust optimization \cite{wiesemann2014distributionally,goh2010distributionally}. 

In \textit{robust optimization} we assume that all possible realizations of the uncertain parameters are contained in a given uncertainty set and the aim is to find a solution which is optimal in the worst-case and feasible for all possible scenarios in the uncertainty set. The robust optimization approach was intensively studied for convex and discrete uncertainty sets; see e.g. \cite{roconv,bentalRobust,bertsimas04priceofrobustness,bertsimas2003combinatorial,kouvelis2013robust,aissi_minmax_survey,buchheimrobust}. Despite its success it can be too conservative since the calculated robust solution has to hedge against all scenarios in the given uncertainty set, which can lead to bad performances in the mean scenario. To overcome this problem several new robust models have been introduced; see e.g. \cite{ben2004adjustable,fischetti2009light,liebchen2009concept,adjiashvili2015bulk}.

In this work we study the \textit{min-max-min robust optimization problem}, which was first introduced in \cite{buchheim2016min} to overcome the conservativeness of the classical robust approach. We consider deterministic optimization problems
\begin{equation}\label{eq:min}\tag{P}
\min_{x\in X} c^\top x
\end{equation}
where $X\subseteq\binvar{n}$ is the set of incidence vectors of all feasible solutions and $c\in\R_+^n$ is an uncertain cost-vector which we assume is contained in a given convex uncertainty set $U\subseteq\R_+^n$. Note that most of the combinatorial problems can be modeled as binary problems and that the restriction to binary variables is theoretically no limitation as general integer variables can be modeled by binary variables using binary encoding. 

Similar to the idea of k-adaptability (\cite{bertsimas2010finite,hanasusanto2015k,subramanyam2019k}) the main idea of the min-max-min robust optimization problem is to hedge against the uncertainty in the cost-vector in a robust manner, i.e. considering the worst-case costs over all scenarios in $U$, while providing more flexibility compared to the classical robust approach since multiple solutions are calculated and can be used to react to the emerging uncertain scenarios. In contrast to the k-adaptability problem the min-max-min robust approach does not consider two-stage structures and is therefore tailored for binary problems where no second-stage decisions exist, i.e. where a set of complete solutions has to be prepared in advance. This can be inevitable in many applications regarding the construction of transportation plans, e.g. evacuation plans for buildings, airports or cities \cite{cepolina2005methodology,rungta2012optimal,mishra2015improved,pillac2016conflict} or route plans for robot systems \cite{skutella2011route}. Note that the restriction to objective uncertainty is still an interesting case since despite many combinatorial problems, also real-world problems appear which can be modeled by objective uncertainty (see e.g. \cite{stienen2021optimal}). Furthermore recent works on two-stage robust optimization indicate that it is possible to use Lagrangian relaxation to shift uncertain constraints into the objective function; see \cite{subramanyam2022lagrangian,lefebvreadaptive}.

 
Since 2015 the min-max-min robust approach was studied for several uncertainty sets and combinatorial problems. It was studied for convex uncertainty sets in \cite{buchheim2016min,buchheim2016conference} and for discrete uncertainty sets in \cite{JannisDiscrete}. Regarding its complexity and solvability the min-max-min robust problem is a very interesting problem due to the unusual connection between its problem parameters and its complexity. A reasonable assumption would be that the problem gets harder to solve with increasing $k$. However this is not true from a theoretical as well as from a computational point of view. While for discrete uncertainty sets it is weakly or strongly NP-hard for most of the classical combinatorial problems \cite{JannisDiscrete}, in the case of convex uncertainty the problem can be solved in polynomial time if $k\ge n$ and if \eqref{eq:min} can be solved in polynomial time \cite{buchheim2016min}. On the other hand it is NP-hard for each fixed $k\in\N$ even if $U$ is a polyhedron. The authors in \cite{buchheim2016min} present an efficient algorithm for the case $k\ge n$ and a fast heuristic for each fixed $k\in \N$. Later in \cite{chassein2020complexity} it was shown that the problem can be solved in polynomial time for several combinatorial problems if $U$ is a convex budgeted uncertainty set and $k=2$. In \cite{ChasseinGKP19} faster exact and heuristic algorithms for the same uncertainty set were presented. For the discrete budgeted uncertainty set the authors in \cite{goerigk2020min} derive exact and fast heuristic algorithms and show that the problem is weakly or strongly NP-hard for most of the classical combinatorial problems. For binary uncertainty sets defined by linear constraints it was shown in \cite{claus2020note} that the min-max-min robust problem is $\Sigma_2^p$-complete. Recently an efficient exact algorithm based on smart enumeration was derived in \cite{arslan2020min} for problems where $X$ does not contain too many good solutions. In \cite{eufinger2020robust} the min-max-min robust problem was applied to the vehicle routing problem where a set of $k$ possible routes has to be prepared in advance which are robust to uncertain traffic scenarios. The idea of the min-max-min robust approach was also applied to the regret robust approach in \cite{crema2020min}.

However, no efficient algorithms are known which solve the min-max-min problem for any $k\in \mathbb N$ efficiently with a certain approximation guarantee. Furthermore nothing is known about the complexity of the problem if $k$ has intermediate size.

As mentioned above the min-max-min problem has a similar structure as the \textit{k-adaptability problem} which was introduced in \cite{bertsimas2010finite} to approximate two-stage robust optimization problems with integer second-stage variables (recently a short note on the incorrectness of the continuity assumption made in the latter paper appeared \cite{kedad2023continuity}). The idea is to calculate $k$ second-stage policies already in the first stage and choose the best of it after the scenario is revealed. This idea provides a heuristic solution to the exact two-stage robust problem and no approximation guarantees were known so far. Due to the similar structure, algorithmic ideas from the k-adaptability literature can also be applied to the min-max-min problem. In \cite{hanasusanto2015k} a mixed-integer programming formulation was derived to solve the k-adaptability problem. Later in \cite{subramanyam2019k} the authors present an algorithm based on a branch~\&~bound scheme which iteratively constructs partitions of scenarios. In \cite{ghahtarani2023double} a logic-based benders decomposition approach is used where the main problem calculates a first-stage solution while a subproblem is used to derive cuts on the objective value. Interestingly the subproblem is a classical min-max-min problem and the authors use methods from the min-max-min literature to solve it. Finally in \cite{bertsimas2011geometric} geometric properties of the uncertainty set are used to derive approximation guarantees for the $k$-adaptability problem with right-hand side uncertainty.

From a theoretical perspective there is not much known about the approximation guarantees the $k$-adaptability approach provides for the two-stage robust problem. The only result into that direction was proved in \cite{hanasusanto2015k} where it is shown that the $k$-adaptability problem provides an optimal solution of the exact two-stage problem if the number of second-stage policies is at least the dimension of the problem. We will close this gap in this work, providing better bounds on $k$ which lead to a certain problem-specific approximation guarantee. Furthermore we provide the first approximation algorithms with a provable approximation guarantee for the $k$-adaptability problem.

\paragraph*{Contributions}
\begin{itemize}
\item  We provide efficient algorithms to calculate solutions with provable additive and multiplicative approximation guarantees for the cases where a) $k$ is smaller but close to $n$ and b) $k$ is a fraction of $n$. The derived guarantees hold for a wide class of binary problems and involve a problem-specific parameter. 
\item We use the approximation guarantees to derive ranges for parameter $k$ for which the algorithm provides a certain approximation guarantee.
\item We show for the first time that under certain assumptions the min-max-min robust problem remains oracle-polynomial solvable if $k=n-l$ and $l$ is a fixed parameter. 
\item We extend the derived approximation guarantees to the $k$-adaptability approach and show that they can be used to calculate better bounds for the number of second-stage policies $k$ which are necessary to achieve a certain approximation guarantee for the exact two-stage robust problem. 
\item We derive an efficient approximation algorithm which calculates solutions for the $k$-adaptability problem with a certain approximation guarantee.
\item We perform experiments to test both approximation algorithms on knapsack and shortest path instances and on a generic $k$-adaptability problem and a two-stage network construction shortest path problem.
\end{itemize}  
The paper is organized as follows. In Section \ref{sec:preliminaries} we provide preliminary results on min-max-min robustness and $k$-adaptability. In Section \ref{sec:kcloseton} we provide the approximation algorithm and derive additive and multiplicative approximation guarantees depending on $n$ and $k$. In Section \ref{sec:k-adaptability} we extend the derived results to the $k$-adaptability approach and provide better bounds on the parameter $k$ to achieve a certain approximation guarantee. Afterwards we present an approximation algorithm and prove its approximation guarantee. Finally in Section \ref{sec:computations} we show the results of our computational experiments and give a conclusion in Section \ref{sec:conclusion}.

\section{Preliminaries}\label{sec:preliminaries}
\subsection{Notation}
We define $[k]:=\left\{ 1,\ldots ,k\right\}$ for $k\in \N$ and $\R_+^n=\{ x\in\R^n: x\ge 0\}$. We denote by $\|x\|:=\sqrt{\sum_{i\in [n]} x_i^2}$ the euclidean norm and by $\|x\|_\infty :=\max_{i\in [n]}|x_i|$ the maximum norm. The convex hull of a finite set $S=\{s^1,\ldots , s^m\}$ is denoted by
\[
\conv{S}= \left\{ s=\sum_{i\in[m]}\lambda_i s^i : \lambda_i\ge 0 \ \forall i\in[m], \ \sum_{i\in[m]}\lambda_i=1\right\}.
\]
The vector of all ones is denoted by $\1$ and the $i$-th unit vector by $e_i$.
\subsection{Min-max-min Robust Optimization}
Formally the min-max-min robust optimization problem is defined as
\begin{equation}\label{eq:minmaxmin}\tag{M$^3(k)$}
\min_{\x{1},\ldots ,\x{k}\in X} \max_{c\in U} \min_{i=1,\ldots ,k} c^\top \x{i}
\end{equation}
where $X\subseteq \{ 0,1\}^n$, $U\subseteq \R_+^n$ is a convex uncertainty set and $k\in\mathbb N$ is a given parameter controlling the conservativeness of the problem. In \cite{buchheim2016min} the authors study Problem \eqref{eq:minmaxmin} for convex uncertainty sets $U$ and show, by using lagrangian relaxation, that Problem \eqref{eq:minmaxmin} for any $k\in\N$ is equivalent to problem
\begin{equation}\label{eq:minmaxreformulation}
\min_{x\in X(k)}\max_{c\in U} c^\top x
\end{equation}
where $X(k)$ is the set of all convex combinations derived by at most $k$ solutions in $X$, i.e.
\[
X(k):=\left\{ x\in\R^n: \ x=\sum_{i\in [k]}\lambda_i\x{i}, \ \x{i}\in X, \ \lambda\in \R_+^k, \ \sum_{i\in [k]}\lambda_i=1\right\} .
\]
By the theorem of Carath\'{e}odory it follows that each point in $\conv{X}$ can be described by a convex combination of at most $n+1$ points in $X$, therefore it holds $X(k)=\conv{X}$ for all $k\ge n+1$. Since for any given point $x\in\conv{X}$ and $\mu\in\R_+$ we have
\[
\max_{c\in U} c^\top (\mu x) = \mu \max_{c\in U} c^\top x
\]
an optimal solution is always attained on the boundary of $\conv{X}$ if $k\ge n+1$, i.e. can be described by a convex combination of at most $n$ solutions in $X$. It follows that for each $k\ge n$ Problem \eqref{eq:minmaxmin} is equivalent to the problem
\begin{equation}\label{eq:minmaxconv}
\min_{x\in \conv{X}}\max_{c\in U} c^\top x .
\end{equation}
From the latter result we obtain the following chain of optimal values, where we denote by $\opt(k)$ the optimal value of $\eqref{eq:minmaxmin}$ with $k$ solutions:
\[
\opt(1)\ge \opt(2)\ge \ldots \ge \opt(n)=\opt(n+1)=\ldots .
\]
Note that $\opt(1)$ is equal to the optimal value of the classical robust problem.

For an optimal solution $x^*$ of Problem \eqref{eq:minmaxconv} the corresponding optimal solution $\x{1},\ldots ,\x{n}$ of Problem (M$^3(n)$) can be calculated in polynomial time, if we can linearly optimize over $X$ in polynomial time; see \cite{grotschel2012geometric,buchheim2016min} for more details.

Problem \eqref{eq:minmaxconv} is a convex problem, since the objective function $f(x):= \max_{c\in U} c^\top x$ is convex and $\conv{X}$ is a convex set. Unfortunately for many classical combinatorial problems no outer-description of polynomial size for $\conv{X}$ is known. Nevertheless the authors in \cite{buchheim2016min} prove that Problem \eqref{eq:minmaxconv} and therefore the min-max-min robust problem can be solved in polynomial time if the underlying deterministic problem \eqref{eq:min} can be solved in polynomial time. More precisely the theorem states that, if we can linearly maximize over $U$ in polynomial time and if we can linearly minimize over $X$ in polynomial time, then we can solve the min-max-min robust problem in polynomial time. However the proof in \cite{buchheim2016min} is not constructive, i.e. no implementable algorithm with a polynomial runtime guarantee is presented. Instead the authors present a column-generation algorithm to solve Problem \eqref{eq:minmaxmin} for $k\ge n$ where iteratively the deterministic problem \eqref{eq:min} and an adversary problem over $U$ is solved. They show that this algorithm is very efficient on random instances of the knapsack problem and the shortest path problem. Furthermore the same algorithm was used successfully in \cite{eufinger2020robust} for the min-max-min version of the capacitated vehicle routing problem. We will adapt this algorithm to find good approximate solutions and calculate strong lower bounds for Problem \eqref{eq:minmaxmin} for any $k$.

On the other hand it is shown in \cite{buchheim2016min} that Problem~\eqref{eq:minmaxmin} with an uncertain constant is NP-hard for any fixed~$k\in\N$, even if~$U$ is a polyhedron given by an inner description and~$X=\{0,1\}^n$. This result fits to the computational experience which was made in other publications, where it turns out that Problem \eqref{eq:minmaxmin} is very hard to solve for small $k\in \N$; see \cite{arslan2020min,ChasseinGKP19}. Nevertheless to tackle the problem even for small $k$ in \cite{buchheim2016min} the authors present an heuristic algorithm which is based on the column-generation algorithm mentioned above. The idea is to solve Problem \eqref{eq:minmaxmin} for $k=n$ with this algorithm and afterwards select $k$ of the calculated solutions with largest induced weights, given by the optimal convex combination of Problem \eqref{eq:minmaxconv}. It is shown computationally that this heuristic calculates solutions which are very close to the optimal value of Problem \eqref{eq:minmaxmin}. We will present a theoretical understanding of this behavior in Section \ref{sec:kcloseton} for the first time and provide approximation bounds.

\subsection{Two-stage robust optimization and k-adaptability}
The k-adaptability problem was introduced with the goal to approximate two-stage robust problems with second-stage integer variables in \cite{bertsimas2010finite}. We define the two-stage robust problem with objective uncertainty as
\begin{equation}\label{eq:two-stage_problem}\tag{2RO}
\min_{x\in X}\max_{\xi\in U}\min_{y\in Y(x)} \ d^\top x + \xi^\top y
\end{equation}
where $X\subseteq \R^m$, $Y(x)\subseteq \{ 0,1\}^n$ for all $x\in X$ and $U\subseteq \R_+^n$ is a convex uncertainty set. The variables $x$ are the first-stage decisions which have to be taken before the uncertainty reveals. The variables $y$ are second-stage decisions which can be taken after the uncertain parameters $\xi$ are known. Note that the feasible set $Y(x)$ depends on the taken first-stage decision. While the following results hold for the more general case, in the linear case often the second-stage feasibility set is given as
\[
Y(x)=\left\{ y\in Y\subseteq\{ 0,1\}^n \ | \ Tx + Wy \le h\right\}
\]
for matrices $T$, $W$ and vector $h$ of appropriate size. The vector $d\in\R_+^m$ contains the given first-stage costs which we assume to be not uncertain, although all the following results are still valid otherwise. 

The $k$-adaptable problem was already studied in \cite{bertsimas2010finite,hanasusanto2015k,subramanyam2019k,ghahtarani2023double}. The idea is to calculate a set of $k$ second-stage solutions $y^1,\ldots ,y^k\in Y(x)$ already in the first-stage and select the best one for each scenario $\xi\in U$. This problem can be formulated as
\begin{equation}\label{eq:k-adaptability}\tag{k-adapt}
\min_{\substack{x\in X, \\ y^1,\ldots , y^k \in Y(x)}}\max_{\xi\in U}\min_{i=1,\ldots ,k} \ d^\top x + \xi^\top y^i .
\end{equation}
Note that while the $k$-adaptability problem is complexity-wise harder than the min-max-min problem \eqref{eq:minmaxmin}, both problems have a quite similar structure. Namely, if we fix a first-stage solution $x$, then \eqref{eq:k-adaptability} becomes a classical min-max-min robust problem \eqref{eq:minmaxmin}.

We can reformulate \eqref{eq:k-adaptability} as
\[
\min_{\substack{x\in X, \\ y^1, \ldots, y^k \in Y(x)}}\max_{\xi\in U}\min_{i=1,\ldots ,k} \ d^\top x + \xi^\top y^i = \min_{\substack{x\in X, \\ y^1, \ldots, y^k \in Y(x)}}\max_{\xi\in U}\min_{\substack{\lambda\ge 0, \\ \sum_i \lambda_i = 1}} \ d^\top x + \sum_{i=1}^{k}\lambda_i\xi^\top y^i .
\]
Since the feasible regions of the inner max and the inner min are both convex, and the objective function is linear in both variables, we can swap the maximum and the inner minimum, which leads to the formulation
\[
\min_{\substack{x\in X, \ \lambda\ge 0 \\ y = \sum_{i\in [k]}\lambda_iy^i \\ \sum_{i\in [k]} \lambda_i = 1 \\ y^1, \ldots, y^k \in Y(x)}}\max_{\xi\in U}\ d^\top x + \xi^\top y .
\]
Note that each feasible solution $y$ of the latter problem is a convex combination of $k$ feasible second-stage policies. Since each point in the convex hull can be described by a convex combination of at most $n+1$ points (by Theorem of Carath\'eodory), it follows  from the last formulation that \eqref{eq:k-adaptability} attains the exact optimal value of \eqref{eq:two-stage_problem} for each $k\ge n+1$ which was proved in \cite{hanasusanto2015k}. In the following we denote by $\opt(2RO)$ the optimal value of the two-stage robust problem \eqref{eq:two-stage_problem}. Furthermore we denote by $\adapt(k)$ the optimal value of \eqref{eq:k-adaptability} with $k$ second-stage policies. Note that by the previous analysis it holds $\opt(2RO)=\adapt(n+1)$. Furthermore we obtain the following chain of optimal values:
\[
\adapt(1)\ge \adapt(2)\ge \ldots \ge \adapt(n)\ge \opt(2RO)=\adapt(n+1)=\adapt(n+2)=\ldots
\]
where $\adapt(1)$ is equal to the optimal value of the classical robust problem.

\section{Approximation algorithm for min-max-min robust optimization}\label{sec:kcloseton}
The min-max-min robust problem is known to be easy to solve if $k\ge n$, both theoretically and practically (see \cite{buchheim2016min}), while it is NP-hard for any fixed $k\in \N$ even if $U$ is a polyhedron and $X=\{0,1\}^n$. Moreover recent results indicate that it is computationally very hard to solve \eqref{eq:minmaxmin} exactly for general convex uncertainty sets $U$ even if $k$ is fixed and small, e.g. $k\in\{2,3,4\}$; see \cite{buchheim2016min,ChasseinGKP19,arslan2020min}. However nothing is known for the intermediate setting, i.e. if $k$ is not small but smaller than $n$. Furthermore no algorithms with approximation guarantees are known for the problem. 

In this section we approach this gap by answering the following research questions.
\begin{enumerate}[RQ1.]
\item Can we derive an algorithm with general additive and multiplicative approximation guarantees for Problem \eqref{eq:minmaxmin} for any number of solutions $k$?
\item For which range of $k$ is a certain approximation guarantee valid? 
\item  What is the theoretical complexity of Problem \eqref{eq:minmaxmin} for an intermediate number of solutions $k$?
\end{enumerate}

We first have to define what we mean by ``intermediate size of $k$''. To this end we consider two cases where a) $k=n-l$ for a fixed $l\in [n-1]$ and b) $k=qn$ for a fixed $q\in (0,1)$. Case a) can be interpreted as ``$k$ is close to $n$'' while case b) means ``$k$ is a fraction of $n$'' where the parameter $q$ controls the distance to $1$ and to $n$. 

In Algorithm \ref{alg:approx} we present an efficient algorithm which calculates feasible solutions for \eqref{eq:minmaxmin}. The algorithm was already studied computationally in \cite{buchheim2016min}. In this work we study the approximation performance of Algorithm \ref{alg:approx}. Note that Step \eqref{step:opt_n} can be solved theoretically in oracle-polynomial time in the input values; see \cite{buchheim2016min}. This step can be implemented by using any algorithm for the case $k\ge n$, e.g. the oracle-based column-generation algorithm presented in \cite{buchheim2016min}. In Step \eqref{step:opt_lambda} we then calculate the optimal coefficients of the convex combination of the calculated $n$ solutions and afterwards select only the $k$ solutions with largest coefficients. The maximum expression in the problem of Step \eqref{step:opt_lambda} can be dualized for classical uncertainty sets (e.g. polyhedra or ellipsoids). Hence this step involves solving a continuous convex optimization problem which can be done in polynomial time for polyhedral or ellipsoidal uncertainty sets. In \cite{buchheim2016min} it was shown that the algorithm is computationally very efficient and provides solutions which are often close to optimal. However no theoretical understanding of this behavior is known. In this section we will, for the first time, derive additive and multiplicative problem-specific approximation guarantees for Algorithm \ref{alg:approx}. Furthermore we will show how to calculate ranges for $k$ for which a given approximation guarantee holds. The derived results can be used to show that \eqref{eq:minmaxmin} is actually oracle-polynomial solvable under certain conditions even if $k<n$.

\begin{algorithm}\caption{Approximation Algorithm for \eqref{eq:minmaxmin}}\label{alg:approx}
	\textbf{Input:} $n\in\N$, $k\in [n]$, convex set $U\subset\R^{n}$, $X\subseteq \{0,1\}^n$
	\begin{algorithmic}[1]
		\State calculate an optimal solution $x^1,\ldots ,x^n$ of (M$^3(n)$) \label{step:opt_n}
		\State calculate an optimal solution $\lambda^*$ of \label{step:opt_lambda}
		\begin{align*}
		\min_{\lambda\ge 0} \ & \max_{c\in U} \ c^\top x \\
		s.t. \quad & \sum_{i=1}^{n} \lambda_i=1 \\
		& x=\sum_{i=1}^{n} \lambda_i x^i 
		\end{align*}
		\State sort the $\lambda^*$ values in decreasing order
		\[
		\lambda_{i1}^*\ge \lambda_{i2}^* \ge \ldots \ge \lambda_{in}^*
		\]
		\State \textbf{Return:} $x^{i1},\ldots ,x^{ik}$
	\end{algorithmic}
\end{algorithm}

\begin{remark}[Variant of Algorithm \ref{alg:approx}]\label{rem:variant_alg1}
Note that we can add a sparsity constraint to the problem in Step \ref{step:opt_lambda} which ensures that at most $k$ of the optimal $\lambda_i$ values are non-zero. More precisely we can solve the following problem in Step \ref{step:opt_lambda} in Algorithm \ref{alg:approx}:
\begin{align*}
		\min \ & \max_{c\in U} \ c^\top x \\
		s.t. \quad & \sum_{i=1}^{n} \lambda_i=1 \\
		& x=\sum_{i=1}^{n} \lambda_i x^i \\
		& \lambda_i\le u_i \quad i=1,\ldots ,n\\
		& \sum_{i=1}^{n} u_i \le k \\
		& \lambda\in \R_+^n, u\in \{ 0,1\}^n
\end{align*}
Since the latter problem selects the best $k$ of the $n$ solutions involved in the convex combination, the solution returned by Algorithm \ref{alg:approx} when using the latter problem has an objective value at most as large as the original solution of Algorithm \ref{alg:approx}. Hence all the approximation results shown in this paper also hold for this variant of the algorithm. On the other hand the latter problem is computationally harder since it involves a sparsity constraint. In our experiments in Section \ref{sec:computations} we show that indeed the solutions of the variant have better objective values coming along with a slightly larger computation time. We will denote the variant of the algorithm as Algorithm \ref{alg:approx} (Variant).
\end{remark}
 
In the following we denote the exact optimal value of Problem \eqref{eq:minmaxmin} by $\opt(k)$ and the objective value of the solution returned by Algorithm \ref{alg:approx} as $\app(k)$. We say Algorithm \ref{alg:approx} has an \textit{additive approximation guarantee} of $a: \mathbb N\to \R_+$ if
\[
\app(k) \le \opt(k) + a(n)
\]
for all $n\in \mathbb N$. We say Algorithm \ref{alg:approx} has an \textit{multiplicative approximation guarantee} of $a: \mathbb N\to \R_+$ if
$$
\app(k) \le (1+a(n))\opt(k) 
$$
for all $n\in \mathbb N$. Note that $a$ can also be a constant function, e.g. $a(n)\equiv \eps$ for any $\eps>0$ and we say the algorithm has a \textit{constant approximation guarantee} in this case.

We assume that we have an oracle which returns an optimal solution of the deterministic problem \eqref{eq:min} for each objective cost vector $c\in U$ in constant time. If we speak in the following of an oracle-polynomial algorithm this means that the calculations for the deterministic problem are assumed to be constant.

We assume that $\|c\|_\infty \le M_\infty$ and $\|c\|_\infty\ge m_\infty>0$ for all $c\in U$. While assuming an upper bound on $U$ is more natural, the assumption on the lower bound seems restricting. However all results in Section \ref{sec:additive_errors} can be derived without assuming a lower bound $m_\infty$ leading to non-substantially worse guarantees. For the multiplicative bounds in Section \ref{sec:multiplicative_errors} the assumption is needed since for non-positive objective values the definition of multiplicative approximation guarantees is not well-defined. From a practical point of view the assumption is reasonable since for many combinatorial problems the objective costs can be assumed to be strictly positive.  Note that considering a maximum-norm bound for any $n\in \N$ is less restrictive than using the euclidean-norm since the latter grows with increasing $n$ even if the entries of the scenarios remain of the same size. In the following we denote the number of non-zero entries of $x\in X$ by $\|x\|_0$, i.e.
\[
\|x\|_0:=|\{i\in [n]: x_i=1\}|.
\]
We furthermore assume that for a given problem class there exist functions $\underline{p},\bar p: \mathbb N \to \R_+$ such that for each instance $X$ of the problem class of dimension $n$ it holds $\underline{p}(n)\le \| x\|_0 \le \bar p(n)$ for all $x\in X$. Note that while $\underline p(n)\equiv 1$ and $\bar p(n)= n$ are always valid functions it is possible to derive problem-specific tighter functions as shown in the following proposition. The problem specific functions $\underline{p},\bar p$ will later appear in the derived approximation guarantees.
\begin{proposition}\label{prop:p-functions_examples}
In the following we present functions $\underline p, \bar p$ for a list of combinatorial problems.
\begin{enumerate}[a)]
\item For the \textit{spanning tree problem} defined on a graph $G=(V,E)$ for each solution $x$ we have $\|x\|_0=|V|-1$, i.e. $\underline p(n)=1$ and $\bar p(n)\le n$. If $G$ is a complete graph, then the number of edges is $|V|(|V|-1)$, i.e. for each given $n$ the number of vertices $|V|$ is uniquely defined and we have $\underline p(n)=\bar p(n)\le \sqrt{n}$.
\item For the \textit{matching problem} defined on a graph $G=(V,E)$ for each solution $x$ we have $\|x\|_0\le \frac{1}{2}|V|$, i.e. $\underline p(n)=1$ and $\bar p(n)\le n$. If $G$ is a complete graph (which is the case for the assignment problem), then the number of edges is at most $|V|^2$, i.e. we have $\underline p(n)= \bar p(n)\le \sqrt{n}$.
\item For each \textit{cardinality constrained} problem, i.e. $X_c=\left\{ x\in X \ | \ \sum_i^n x_i=\bar p\right\}$ for a given fixed $\bar p\in [n]$, we have $\underline p(n)=\bar p(n) = \bar p$. One popular example from robust combinatorial optimization is the $p$-selection problem where $X=\{ 0,1\}^n$.
\item For the \textit{traveling salesmen problem} (TSP) defined on a complete graph $G=(V,E)$ for each solution $x$ we have $\|x\|_0 = |V|$, i.e. $\underline p(n)=\bar p(n)\le \sqrt{n}$.
\item For the \textit{vehicle routing problem} (VRP) defined on a complete graph $G=(V,E)$ with $m$ customers and one depot for each solution $x$ we have $\|x\|_0 \le 2m$ (in case each vehicle visits exactly one customer) and $m+1\le \|x\|_0$ (in case only one vehicle visits all customers). Since the graph has $m(m+1)$ edges we have $\bar c \sqrt{n} \le \underline p(n)\le\bar p(n)\le \bar C\sqrt{n}$ for constants $\bar c,\bar C\ge 0$. 
\end{enumerate}
\end{proposition}

\subsection{Additive approximation guarantees}\label{sec:additive_errors}
We first prove the following general lemma which generalizes the result used in the proof of Theorem 6 in~\cite{buchheim2016min}. 

\begin{lemma}\label{lem:bound_opts_optk}
Assume $s,k\in [n]$ where $s<k$, then it holds
\[
\opt(s)-\opt(k)\le M(n) \frac{k-s}{s+1}.
\]
where $M(n):=M_\infty \bar p(n)-m_\infty \underline{p}(n)$.
\end{lemma}
\begin{proof}
Let $x^*(k)$ be an optimal solution of Problem \eqref{eq:minmaxreformulation} with parameter $k$. Then by the results in \cite{buchheim2016min} we have
\[
\opt(k) = \max_{c\in U}c^\top x^*(k)
\]
and there exists a convex combination $x^*(k)=\sum_{i\in[k]}\lambda_i x^i$ where $x^i\in X$ and $\lambda\in\R_+^k$ with $\sum_{i\in[k]}\lambda_i=1$. We may assume without loss of generality that $\lambda_1\ge \ldots \ge \lambda_k$. Define a solution $x(s)$ by
\begin{equation}\label{eq:define_s_solution}
x(s):=\sum_{i\in [s-1]} \lambda_i x^i + \left(\sum_{i=s}^{k}\lambda_i\right) x^s,
\end{equation}
then $x(s)\in X(s)$ and therefore $\opt(s)\le \max_{c\in U}c^\top x(s)$. Furthermore let $c^*(s)\in \argmax_{c\in U}c^\top x(s)$. It follows
\begin{align*}
\opt(s)-\opt(k) & \le \max_{c\in U}c^\top x(s) - \max_{c\in U}c^\top x^*(k) \\
& \le c^*(s)^\top \left( x(s)-x^*(k)\right) \\
& = \sum_{i=s+1}^{k}\lambda_i c^*(s)^\top \left( x^s-x^i\right)\\
& \le M(n) \left( \sum_{i=s+1}^{k}\lambda_i \right)
\end{align*}
where the second inequality holds since $c^*(s)$ is a subgradient of the function $g(x)=\max_{c\in U}c^\top x$ in  $x(s)$. For the first equality we used the definition of $x(s)$ and $x^*(k)$ and for the last inequality we used the assumption $\underline{p}(n)\le \|x\|_0 \le \bar p(n)$ and $m_\infty \le c^*(s)_j\le M_\infty$ for all $j\in [n]$. Due to the sorting $\lambda_1\ge \ldots \ge \lambda_k$ and since the sum over all $\lambda_i$ is one, we have $\lambda_i\le \frac{1}{i}$ and hence
\begin{align*}
M(n) \left( \sum_{i=s+1}^{k}\lambda_i \right) &\le M(n) \left( \sum_{i=s+1}^{k}\frac{1}{i} \right)\\
& \le M(n) \frac{k-s}{s+1}
\end{align*}
which proves the result.
\end{proof}
Note that the bound derived in the latter lemma is small if $s$ is large compared to the difference $k-s$. Particularly it holds
$$\opt(k)-\opt(k+1)\le M(n) \frac{1}{k+1} \to 0 \ \text{ for } k\to \infty.$$
We can conclude that the objective value gain when $k$ is increased by one unit decreases to zero when $k$ gets large.
From Lemma \ref{lem:bound_opts_optk} we can easily derive an additive approximation guarantee for Algorithm \ref{alg:approx}.
\begin{corollary}\label{cor:approx_additive}
For given $k$ Algorithm \ref{alg:approx} returns a solution to Problem \eqref{eq:minmaxmin} with
\[
\app(k)\le \opt(k') + M(n) \frac{n-k}{k+1}.
\]
for all $k'\ge k$.
\end{corollary}
\begin{proof}
We can apply Lemma \ref{lem:bound_opts_optk} for $k=n$ and $s=k$. The solution $x(s)$ constructed in \eqref{eq:define_s_solution} is composed of the same solutions $x^1,\ldots ,x^s$ which Algorithm \ref{alg:approx} returns. Hence the approximation guarantee holds also for this solution and we obtain
\[
\opt(k) - \opt(n) \le \app(k)-\opt(n) \le M(n) \frac{n-k}{k+1}.
\]
It always holds $\app(k)-\opt(k')\le \app(k)-\opt(n)$ for all $k'\ge k$ which proves the result.
\end{proof}
Note that the latter approximation guarantee holds for a large number of problem classes and the parameter $M(n)$ is problem-specific since it involves the functions $\underline p,\bar p$. While the approximation guarantee is linear in $n$ and hence can be large, the interesting observation here is that it goes to zero if $k$ approaches $n$.

In Figure \ref{fig:plotAdditiveGuarantee} we show examples of the additive guarantee provided in the latter corollary for different constants $M(n)$. It can be seen that for all relevant functions $\underline p,\bar p$ with $1\le \underline p(n),\bar p(n) \le n$ the approximation guarantee goes to zero when $k$ approaches $n$. The smaller $\underline p, \bar p$ grow the faster the approximation guarantee approaches zero. This indicates that we can approximate \eqref{eq:minmaxmin} well by Algorithm \ref{alg:approx} even for intermediate values $k$.
\begin{figure}
\centering
\includegraphics[scale=0.6]{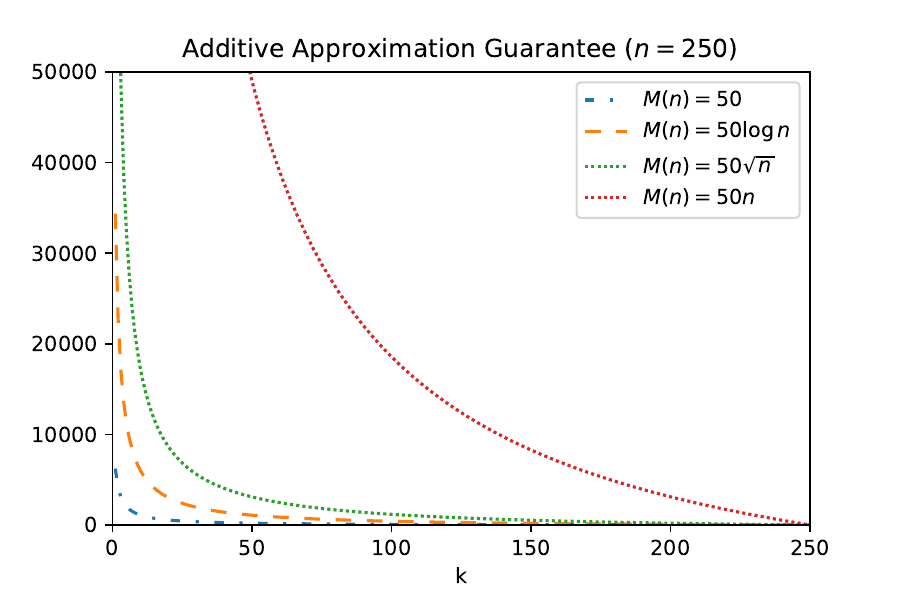}
\caption{Additive approximation guarantee depending on $k$ for fixed $n=250$.}
\label{fig:plotAdditiveGuarantee}
\end{figure}

\paragraph*{Ranges for $k$ with given approximation guarantee}
The latter results lead to the intuition, that the range for $k$ for which Algorithm \ref{alg:approx} provides an approximation guarantee of $a(n)$, gets larger with increasing $n$. This is shown in the following lemma.

\begin{lemma}\label{lem:range_l}
Let $k=n-l$ and $l\in [n-1]$, then Algorithm \ref{alg:approx} returns a solution with approximation guarantee at most $a(n)$ for all 
\[
l \in \left[0, \min\left\{n-1,\frac{a(n)n}{M(n)+a(n)}\right\} \right]\cap \mathbb N.
\]
\end{lemma}
\begin{proof}
We can estimate
\[
\app(k)-\opt(k) \le M(n)\frac{l}{n-l}\le M(n)\frac{a(n)n}{M(n)+a(n)}\frac{M(n)+a(n)}{nM(n)} = a(n)
\]
where the first inequality follows from Corollary \ref{cor:approx_additive} and the second follows from 
\[l\le \frac{a(n)n}{M(n)+a(n)}.\]
\end{proof}
Note that the upper bound term $\frac{a(n)n}{M(n)+a(n)}$ for $l$ in Lemma \ref{lem:range_l} (and therefore the absolute number of approximable $k<n$) goes to infinity if the approximation guarantee grows at least as fast as $M(n)$, i.e. either $\frac{a(n)}{M(n)}\to \infty$ or $\frac{a(n)}{M(n)}\to C\in \R$ for $n\to \infty$. Additionally if $a(n)$ is constant, the upper bound goes to infinity if $M(n)$ grows sub-linearly, i.e. $\frac{n}{M(n)}\to \infty$ for $n\to \infty$. Note that the latter is the case for all examples in Proposition \ref{prop:p-functions_examples}. This means that the range for $k$ for which $\opt(k)$ can be approximated grows with $n$. Additionally this means that the difficult instances, which do not achieve the approximation guarantee are the ones for $k<n-l$. In Figure \ref{fig:range_l} we show examples of the development of the ranges of $k$ for which the approximation guarantee is ensured. It can be seen that the larger the approximation guarantee the smaller is the smallest $k$ that can be approximated. Furthermore for increasing $n$ the range of approximable $k$ increases. 
\begin{figure}\label{fig:range_l}
\centering
\includegraphics[scale=0.45]{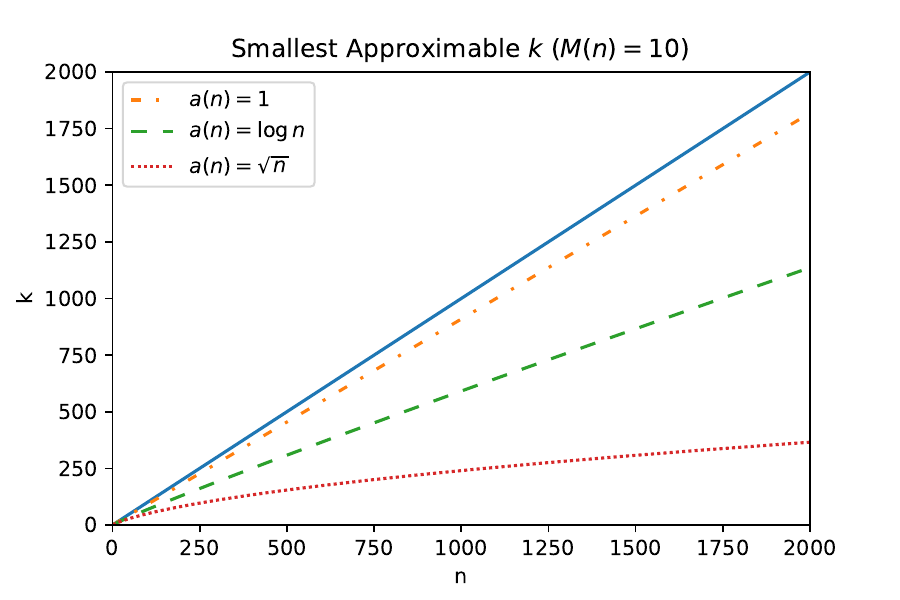} \quad
\includegraphics[scale=0.45]{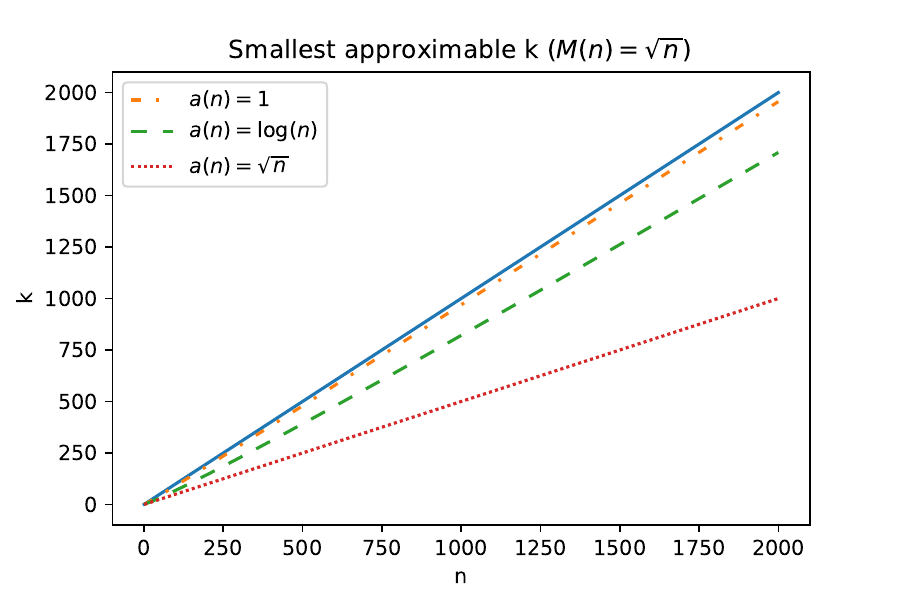}
\caption{The smallest approximable $k=n-l$ for given approximation factor $a(n)$ and $M(n)=10$ (left) or $M(n)=\sqrt{n}$ (right). The blue permanent line is the function $f(n)=n$, hence all $k$ values between this line and the $k=n-l$ line are approximable within the given approximation guarantee.} 
\end{figure}

In Table \ref{tbl:ranges_k} we show analytical lower bounds $k_0$ such that for all $k\ge k_0$ the given approximation guarantee is achieved. The values are derived as follows: first we estimate 
\[
M(n)=M_\infty\bar p(n) - m_\infty\underline p(n)\le M_\infty\bar p(n).
\]
Using this estimation we obtain that the result of Lemma \ref{lem:range_l} is valid for all
\[
k\ge n- \frac{a(n)n}{M_\infty\bar p(n)+a(n)}.
\]
Note that while the estimation of $M(n)$ results in easier notation, the size of this value significantly influences the quality of the bounds. Hence using the exact value for $M(n)$ can lead to better bounds. The values in Table \ref{tbl:ranges_k} are derived by plugging in $\bar p(n)$ and $a(n)$ and bounding the derived terms from above. It is important to note that for each value $k_0$ in the table, it is given as a fraction of $n$, i.e. $k_0=q(n)n$ where $q(n)\in (0,1)$. Depending on the combination of $a(n)$ and $\bar p(n)$ either $q(n)$ is constant, $q(n)\to 0$ or $q(n)\to 1$ for $n\to \infty$. Note that a constant $q(n)$ means that the absolute number of $k$ values for which the approximation holds grows with $n$ but remains always a certain fraction of $n$. The case $q(n)\to 0$ is the most desirable case, since here the fraction of approximable $k$ even increases. The less desirable case is the one where $q(n)\to 1$ since it means that the fraction of $n$ which can be approximated goes to zero with increasing $n$. However for most bounds it holds that the absolute number of approximable $k$, i.e. $|[n-k_0, n]\cap \mathbb N|$, goes to infinity for $n\to \infty$.    

Note that most of the problems considered in Proposition \ref{prop:p-functions_examples} have either $\bar p(n)=v$, $\bar p(n)=\sqrt{n}$ or $\bar p(n)=n$.

\begin{table}[h!]
\centering
\begin{tabular}{c||l|l|l}
\diagbox{$\bar p(n)$}{$a(n)$} & $\eps$ & $\log n$ & $n^{1-\gamma}$ \\
\hline
const. $v$ & $n\left( 1-\frac{\eps}{M_\infty v+\eps}\right)$ & $n\left(1-\frac{\log(n)}{M_\infty v + \log(n)}\right)$ & $n\left( 1 - \frac{n^{1-\gamma}}{M_\infty v+n^{1-\gamma}}\right)$\\
$\log{n}$ & $n\left( 1-\frac{\eps}{(M_\infty+1)\log{n}}\right)$ & $n\left( 1-\frac{1}{M_\infty +1}\right)$ & $n\left( 1 - \frac{n^{\frac{1}{2}-\gamma}}{M_\infty + n^{\frac{1}{2}-\gamma}}\right)$  \\
$n^{1-\delta}$ & $n\left( 1 - \frac{\eps}{n^{1-\delta}(M_\infty + 1)} \right)$ & $n\left( 1 - \frac{\log n}{n^{1-\delta}(M_\infty + 1)}\right)$ & $n\left( 1 - \frac{n^{\delta-\gamma}}{M_\infty + n^{\delta-\gamma}}\right)$
\end{tabular}
\caption{For each pair of function $\bar p(n)$ and approximation guarantee $a(n)$ the table shows a value $k_0$ such that for all $k\ge k_0$ Algorithm \ref{alg:approx} returns a solution of \eqref{eq:minmaxmin} with approximation guarantee $a(n)$. It holds $\gamma,\delta\in [0,1)$.}
\label{tbl:ranges_k}
\end{table}

\paragraph*{Complexity}
We will now analyze the complexity of \eqref{eq:minmaxmin} for $k=n-l$ where $l\in [n-1]$ is a fixed parameter. The following lemma shows, that for a fixed $l$ we can find a value $n_0$ such that for all $n\ge n_0$ Algorithm \ref{alg:approx} returns an optimal solution (up to an additive accuracy of $\eps>0$) if $\bar p(n)$ grows sub-linear.
\begin{lemma}\label{lem:lowerbound_n}
Assume $\bar p(n)\le Cn^{1-\delta}$ where $\delta\in (0,1]$ and $C>0$. Furthermore let $k=n-l$ and $\eps>0$. If 
\begin{equation}\label{eq:lowerbound_n}
n\ge l^{\frac{1}{\delta}}\left( \frac{CM_\infty}{\eps}+1\right)^{\frac{1}{\delta}}
\end{equation}
then $\opt(n)\le \opt(n-l)\le \opt(n)+\eps$.
\end{lemma}
\begin{proof}
The inequality $\opt(n)\le \opt(n-l)$ follows since \eqref{eq:minmaxmin} attains smaller optimal values if more solutions are allowed. To prove the inequality $\opt(n-l)\le \opt(n)+\eps$ we apply Lemma \ref{lem:bound_opts_optk} with $\bar p(n)=Cn^{1-\delta}$, $\underline{p}(n)\equiv 0$, $s=n-l$ and $k=n$ and we obtain
\begin{align*}
\opt(n-l)-\opt(n)&\le Cn^{1-\delta}M_\infty \frac{l}{n-l+1} \\
& = C M_\infty \frac{l}{n^{\delta}-\frac{l}{n^{1-\delta}}+\frac{1}{n^{1-\delta}}} \\
& \le C M_\infty \frac{l}{n^{\delta}-l} \\
& \le \eps
\end{align*}
where in the second inequality we used $n\ge 1$ and in the third inequality we used inequality~\eqref{eq:lowerbound_n}.
\end{proof}
The result of the latter lemma essentially says that, if we want to solve \eqref{eq:minmaxmin} with $k=n-l$, we do not have to worry about instances where $n$ is large, since we can solve these instances efficiently by using Algorithm \ref{alg:approx}. This is a pretty surprising result since we only have to care about the instances with bounded $n$. If we fix all parameters of the bound in \eqref{eq:lowerbound_n}, then we can also solve \eqref{eq:minmaxmin} for small $n$ by enumerating all solutions, which leads to the following theorem.
\begin{theorem}\label{thm:polynomial_additiveerror}
Assume $\bar p(n)\le Cn^{1-\delta}$ where $\delta\in (0,1]$ and $C>0$. Furthermore let $k=n-l$ and $\eps>0$. Then we can calculate an optimal solution of \eqref{eq:minmaxmin} (up to an additive error of $\eps$) in polynomial time in $n$, if all parameters $M_\infty$, $\delta$, $C$, $\eps$ and $l$ are fixed.
\end{theorem}
\begin{proof}
Given an instance of \eqref{eq:minmaxmin} we check if condition \eqref{eq:lowerbound_n} is true or not, which can be done in polynomial time, since the right hand side is a constant. Note that to avoid calculations of the root-terms an easier bound without the exponents $\frac{1}{\delta}$ could be used.

\textbf{Case $1$:} If the condition is true, we solve $M^3(n)$ up to an additive error of $\eps$ in polynomial time which can be done by Algorithm \ref{alg:approx} due to Lemma \ref{lem:lowerbound_n} and Corollary \ref{cor:approx_additive}.

\textbf{Case $2$:} If condition \eqref{eq:lowerbound_n} is not true, then $n$ is bounded from above by the constant \[\tau:=l^{\frac{1}{\delta}}\left( \frac{2CM_\infty}{\eps}+1\right)^{\frac{1}{\delta}}.\]
Therefore the number of possible solutions in $X$ is in $\mathcal O(2^\tau)$ and hence the number of solutions of \eqref{eq:minmaxmin} is in $\mathcal O(2^{\tau k})=\mathcal O(2^{\tau^2})$ since $k\le n\le \tau$. We can calculate an optimal solution in this case by enumerating all possible solutions of \eqref{eq:minmaxmin} and comparing the objective values. Note that we can calculate the objective value of a given solution $\x{1},\ldots ,\x{k}$ by solving the problem
\[
\min_{x\in \conv{\x{1},\ldots ,\x{k}}} \max_{c\in U} c^\top x
\]
which can again be done in polynomial time in $n$.
\end{proof}

The following corollary follows directly from Theorem \ref{thm:polynomial_additiveerror} and Proposition \ref{prop:p-functions_examples}.
\begin{corollary}\label{cor:additiveerror_combprobs}
Under the assumptions of Theorem \ref{thm:polynomial_additiveerror} we can calculate an optimal solution of \eqref{eq:minmaxmin} (up to an additive error of $\eps$) in oracle-polynomial time, if the underlying problem \eqref{eq:min} is the spanning-tree problem on complete graphs, the matching problem on complete graphs, any cardinality constrained combinatorial problem, the TSP or the VRP. 
\end{corollary}
Note that the necessary condition on $\bar p$ is likely to be true for many other problems. Besides the $p$-selection problem several combinatorial problems with cardinality constraint $|x|=p$ were studied in the literature; see \cite{bruglieri2006annotated} for an overview. Finally note that the latter corollary states that we obtain an oracle-polynomial algorithm (and not a polynomial algorithm) which is because some of the stated problems are NP-hard and hence no polynomial algorithm can be derived unless $P=NP$.

\subsection{Multiplicative approximation guarantees}\label{sec:multiplicative_errors}
In this subsection we follow similar ideas as in the latter section to derive multiplicative approximation guarantees. In contrast to the previous subsection we assume now that $k=\lceil qn \rceil$ for a fixed $q\in (0,1)$. However all results also go through for the case $k=n-l$, leading to slightly different bounds and guarantees. With slight abuse of notation in the following we write $k=qn$ and assume that $k$ is integer.

\paragraph*{Approximation Guarantees}
The following lemma is the multiplicative counterpart of Lemma~\ref{lem:bound_opts_optk}.
\begin{lemma}\label{lem:multiplicative_bound_opts_optk}
Let $s,k\in [n]$ where $s<k$, then it holds
\[
\opt(s)\le \left( 1 + \tilde M(n)\frac{k-s}{s+1}\right)\opt(k)
\]
where $\tilde M(n):=\frac{M_\infty}{m_\infty}\frac{\bar p(n)}{\underline{p}(n)}$. Furthermore if for all instances of dimension $n$ it holds $\underline p(n)=\bar p(n)$, then $\tilde M(n):=\frac{M_\infty}{m_\infty}$ is independent of $n$.
\end{lemma}
\begin{proof}
First note that due to the assumptions $\underline{p}\le \|x\|_0$ for all $x\in X$ and $c\ge m_\infty\1>0$ for all $c\in U$ we have $\opt(k)>0$ for all $k\in[n]$ and therefore the multiplicative approximation guarantee stated in the lemma is well defined. 

By Lemma \ref{lem:bound_opts_optk} we have
\begin{equation}\label{eq:LemmaPTASInequ}
\opt(s)\le \opt(k) + M(n) \frac{k-s}{s+1}\le \opt(k) + M_\infty \bar p(n) \frac{k-s}{s+1}
\end{equation}
Let $x^*(k)$ be an optimal solution of Problem \eqref{eq:minmaxreformulation} with parameter $k$. Then by the results in \cite{buchheim2016min} we have
\[
\opt(k) = \max_{c\in U}c^\top x^*(k)
\]
and there exists a convex combination $x^*(k)=\sum_{i\in[k]}\lambda_i x^i$ where $x^i\in X$ and $\lambda\in\R_+^k$ with $\sum_{i\in[k]}\lambda_i=1$. Then it holds
\begin{equation}\label{eq:LemmaPTASInequ2}
\opt(k)= \max_{c\in U}c^\top x^*(k) \ge m_{\infty}\1^\top x^*(k)
\end{equation}
where the last inequality holds due to the assumption $c\ge m_{\infty}\1$ and since $x^*(k)\ge 0$. We can now reformulate
\begin{equation}\label{eq:boundbelowbybarp}
m_{\infty}\1^\top x^*(k) = \sum_{i\in[k]}\lambda_i m_{\infty}\1^\top x^i \ge \sum_{i\in[k]}\lambda_i m_{\infty}\underline{p}(n) = m_{\infty}\underline{p}(n)
\end{equation}
where in the first inequality we used the assumption $\underline{p}(n)\le \|x\|_0$. Together with \eqref{eq:LemmaPTASInequ2} we obtain $\opt(k)\ge m_{\infty}\underline{p}(n)$. It follows
\[
\frac{\opt(s)}{\opt(k)} \le \frac{\opt(k) + M_\infty \bar p(n) \frac{k-s}{s+1}}{\opt(k)} \le 1 + \frac{M_\infty \bar p(n) \frac{k-s}{s+1}}{m_\infty \underline p(n)}
\]
which proves the result.

The second result follows directly from $\underline p(n)=\bar p(n)$.
\end{proof}
Analogously to Corollary \ref{cor:approx_additive} we can now derive a multiplicative approximation guarantee for each $k\in [n]$ from Lemma \ref{lem:multiplicative_bound_opts_optk}. 

\begin{corollary}\label{cor:approx_multiplicative}
For given $k\in [n]$ Algorithm \ref{alg:approx} returns a solution to Problem \eqref{eq:minmaxmin} with multiplicative approximation guarantee
\[
\app(k)\le \left( 1 + \tilde M(n)\frac{n-k}{k+1}\right)\opt(k')
\]
for all $k'\ge k$.
\end{corollary}
Note that the behavior of the multiplicative approximation guarantee $a(n)=\tilde M(n)\frac{n-k}{k+1}$ is  up to the factor $\tilde M(n)$ the same as for the additive approximation guarantee; see Figure \ref{fig:plotAdditiveGuarantee} for exemplary behaviors.

If we choose $k$ to be a fraction of $n$, i.e. $k=\lceil qn\rceil $ for a fixed $q\in (0,1)$, and if additionally we have $\underline p(n)=\bar p(n)$ (which is the case for the spanning tree problem and the matching problem on complete graphs, the traveling salesmen problem and each cardinality constrained problem), then the approximation factor is given as
\[
1+ \frac{M_\infty}{m_\infty} \frac{n-\lceil qn\rceil}{\lceil qn\rceil +1}\le 1 + \frac{M_\infty}{m_\infty} \frac{(1-q)n}{qn} = 1 + \frac{M_\infty}{m_\infty} \frac{1-q}{q}
\]
which is independent of the dimension $n$. This proves the following corollary.
\begin{corollary}\label{cor:constant_approx}
Algorithm \ref{alg:approx} has a constant multiplicative approximation guarantee for \eqref{eq:minmaxmin} if $k=\lceil qn\rceil$ for a fixed $q\in (0,1)$ and $\underline p(n) = \bar p(n)$.
\end{corollary}

Note that for budgeted uncertainty sets where each parameter is allowed to deviate from its mean by at most $50\%$ we have $\frac{M_\infty}{m_\infty}=1.5$. If we choose $q=\frac{1}{2}$ the constant approximation factor of Algorithm \ref{alg:approx} is 
\[
1 + \frac{M_\infty}{m_\infty} \frac{1-q}{q} = 2.5.
\]

\paragraph*{Ranges for $k$ with given approximation guarantee}
The latter results lead to the intuition, that the range for $k$ for which Algorithm \ref{alg:approx} provides a multiplicative approximation guarantee of $a(n)$, gets larger with increasing $n$. This is shown in the following lemma.

\begin{lemma}\label{lem:range_q}
Let $k=qn$ and $q\in (0,1]$, then Algorithm \ref{alg:approx} returns a solution with multiplicative approximation guarantee at most $a(n)$ for all 
\[
q \in \left[\frac{\tilde M(n)}{\tilde M(n)+a(n)}, 1\right].
\]
\end{lemma}
\begin{proof}
We can estimate
\begin{align*}
&\frac{\app(k)}{\opt(k)} \le 1 +  \tilde M(n)\frac{n-qn}{qn}
\le 1 +  \tilde M(n)\frac{1-q}{q} \\ &\le 1 + \tilde M(n)\frac{a(n)}{\tilde M(n) + a(n)}\frac{\tilde M(n) + a(n)}{\tilde M(n)} = 1+a(n)
\end{align*}
where the first inequality follows from Corollary \ref{cor:approx_multiplicative} and the third follows from 
\begin{equation}\label{eq:lowerBound_q} q\ge \frac{\tilde M(n)}{\tilde M(n)+a(n)}.
\end{equation}
\end{proof}
Note that the lower bound term \eqref{eq:lowerBound_q} for $q$ goes to zero if the approximation guarantee grows faster than $\tilde M(n)$ while it goes to one if it grows slower. If $\tilde M(n)$ grows as fast as $a(n)$, then it is constant. In Figure \ref{fig:range_q} we show examples for the development of the ranges for $k$ for which the approximation guarantee is ensured.

\begin{figure}
\centering
\includegraphics[scale=0.45]{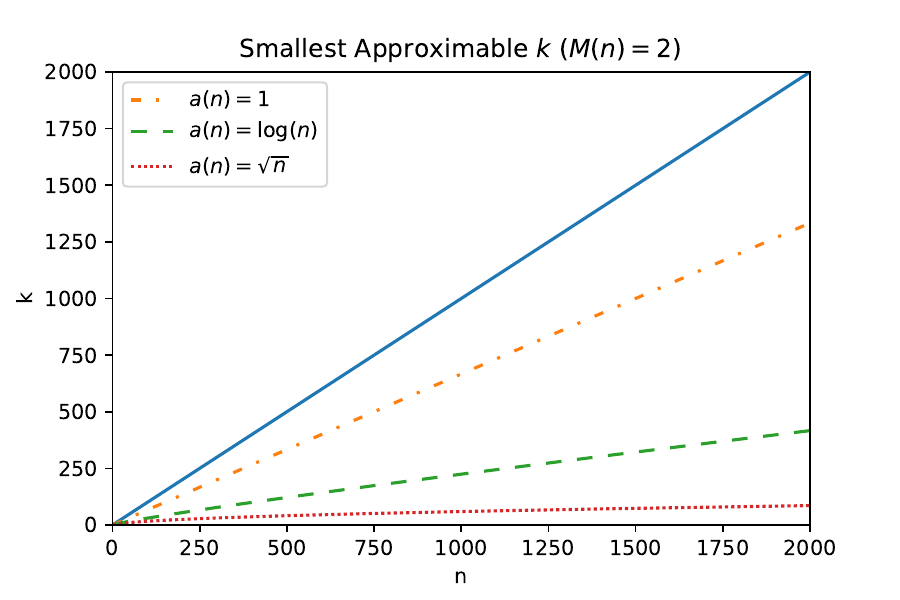} \quad
\includegraphics[scale=0.45]{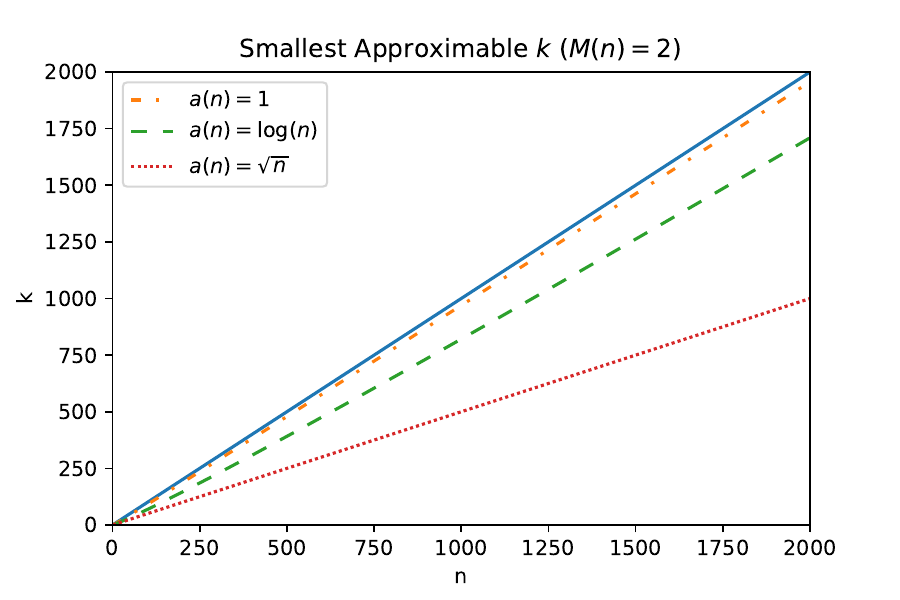}
\caption{The smallest approximable $k=qn$ for given approximation factor $a(n)$ and $\tilde M(n)=2$ (left) or $\tilde M(n)=\sqrt{n}$ (right). The blue permanent line is the function $f(n)=n$, hence all $k$ values between this line and the $k=qn$ line are approximable within the given approximation guarantee.}
\label{fig:range_q}
\end{figure} 

In Table \ref{tbl:ranges_k_multiplicative} we show lower bounds $k_0$ such that for all $k\ge k_0$ the given approximation guarantee is achieved. The values are derived analogously to Table \ref{tbl:ranges_k}. It is important to note that each value $k_0$ in the table is given as a fraction of $n$, i.e. $k_0=q(n)n$ where $q(n)\in (0,1)$. Depending on the combination of $a(n)$ and $\frac{\bar p(n)}{\underline p(n)}$ either $q(n)$ is constant, $q(n)\to 0$ or $q(n)\to 1$ for $n\to \infty$. Note that a constant $q(n)$ means that the number of $k$ values for which the approximation holds grows with $n$ but remains always a certain fraction of $n$. The case $q(n)\to 0$ is the most desirable case, since here the number of $k$ values for which the approximation holds grows with $n$ and the fraction of values for which it holds even increases. The less desirable case is the one where $q(n)\to 1$ since it means that the fraction of $n$ which can be approximated goes to zero with increasing $n$. However for all bounds the absolute number of approximable $k$ values, i.e. $|[n-k_0,n]\cap \mathbb N|$ goes to infinity for $n\to \infty$.    
Note that in contrast to the results in Section \ref{sec:additive_errors} here the function ratio $\frac{\bar p(n)}{\underline p(n)}$ has to be bounded. As shown in Proposition \ref{prop:p-functions_examples} this fraction is constant for the spanning tree problem and the matching problem on complete graphs, as well as for all other mentioned problems.

\begin{table}[h!]
\centering
\begin{tabular}{c||l|l|l}
\diagbox{$\frac{\bar p(n)}{\underline p(n)}$}{$a(n)$} & $\eps$ & $\log n$ & $n^{1-\gamma}$ \\
\hline
const. $v$ & $n\left( \frac{\hat m v}{\hat m v+\eps}\right)$ & $n\left(\frac{\hat m v}{\hat m v+\log n}\right)$ & $n\left(  \frac{\hat m v}{\hat m v+n^{1-\gamma}}\right)$\\
$\log{n}$ & $n\left( \frac{\hat m\log n}{\hat m \log n+\eps}\right)$ & $n\left(\frac{\hat m}{\hat m +1}\right)$ & $n\left( \frac{\hat m\log n}{\hat m\log n + n^{1-\gamma}}\right)$  \\
$n^{1-\delta}$ & $n\left( \frac{\hat m n^{1-\delta}}{\hat m n^{1-\delta}+\eps}\right)$ & $n\left( \frac{\hat m n^{1-\delta}}{\hat m n^{1-\delta}+\log n}\right)$ & $n\left( \frac{\hat m}{\hat m + n^{\delta-\gamma}}\right)$
\end{tabular}
\caption{For each pair of function ratio $\frac{\bar p(n)}{\underline p(n)}$ and approximation guarantee $a(n)$ the table shows a value $k_0$ such that for all $k\ge k_0$ Algorithm \ref{alg:approx} returns a solution of \eqref{eq:minmaxmin} with approximation guarantee $a(n)$. It holds $\hat m:=\frac{M_\infty}{m_\infty}$ and $\gamma,\delta\in [0,1)$.}
\label{tbl:ranges_k_multiplicative}
\end{table}

\paragraph*{Complexity}
Unfortunately following the previous methods it is not possible to prove a multiplicative version of Theorem \ref{thm:polynomial_additiveerror} which provides a polynomial time algorithm to approximate \eqref{eq:minmaxmin} for $k=qn$ for each fixed $q\in (0,1)$. However a weaker version of the result will be proved in the following.
\begin{lemma}\label{lem:multiplicative_lowerbound_n}
Assume that $\underline p (n) = \bar p (n)$ and let $\eps>0$, $k=qn$. If
\[q\in \left[ \frac{\frac{M_\infty}{m_\infty}}{\eps + \frac{M_\infty}{m_\infty}}, 1\right] \] 
then $\opt(qn)\le (1+\eps)\opt(n)$.
\end{lemma}
\begin{proof}
We can apply Lemma \ref{lem:multiplicative_bound_opts_optk} with $k=n$ and $s=qn$ which leads to the bound 
\[
\opt(qn)\le \left( 1 + \tilde M(n)\frac{1-q}{q}\right)\opt(n).
\]
where $\tilde M(n)=\frac{M_\infty}{m_\infty}$ since $\underline p (n) = \bar p (n)$. We can estimate 
\[
\tilde M(n)\frac{1-q}{q} \le \frac{M_\infty}{m_\infty} \frac{\eps}{\eps + \frac{M_\infty}{m_\infty}}\frac{\eps + \frac{M_\infty}{m_\infty}}{\frac{M_\infty}{m_\infty}} = \eps
\]
\end{proof}
Note that the latter result only holds if $\underline p(n) = \bar p(n)$. However this is the case for all problems in Proposition \ref{prop:p-functions_examples} on complete graphs. Furthermore we do not need any assumption on the behavior of $\bar p(n)$ here. The following theorem follow directly from Lemma \ref{lem:multiplicative_lowerbound_n}.
  
\begin{theorem}\label{thm:polynomial_ptas}
Assume that for all $n\in\N$ it holds $\underline{p}(n)=\bar p(n)$, let $\eps>0$ and $k=qn$ with 
\[q\in \left[ \frac{\frac{M_\infty}{m_\infty}}{\eps + \frac{M_\infty}{m_\infty}}, 1\right].\] 
Then \eqref{eq:minmaxmin} can be approximated with an multiplicative error of $1+\eps$ in oracle-polynomial time in $n$ if all parameters $\frac{M_\infty}{m_\infty}$ and $\eps$ are fixed.
\end{theorem}
The next corollary follows directly from Theorem \ref{thm:polynomial_ptas} and Proposition \ref{prop:p-functions_examples}.
\begin{corollary}\label{cor:PTAS_combprobs}
Under the assumptions of Theorem \ref{thm:polynomial_ptas} Algorithm \ref{alg:approx} calculates a solution of \eqref{eq:minmaxmin} with multiplicative approximation guarantee $1+\eps$ in oracle-polynomial time for the spanning-tree problem on complete graphs, the matching problem on complete graphs, any cardinality constrained combinatorial problem, the TSP or the VRP.
\end{corollary}

\section{Approximation algorithm for k-adaptability problems}\label{sec:k-adaptability}
In this section we extend the results of the previous section to the $k$-adaptability problem and answer the following research questions:
\begin{enumerate}[RQ1.]
\item Can we derive general additive and multiplicative approximation guarantees for the k-adaptability problem depending on $k$?
\item For which range of $k$ does the $k$-adaptability problem provide a certain approximation guarantee for the two-stage robust problem? 
\item  Can we derive an algorithm which efficiently calculates a solution of the $k$-adaptability problem with a certain approximation guarantee?
\end{enumerate}

We already showed in Section \ref{sec:preliminaries} that the optimal values of the $k$-adaptability problem follow the chain
\[
\adapt(1)\ge \adapt(2)\ge \ldots \ge \adapt(n)\ge \opt(2RO)=\adapt(n+1)=\adapt(n+2)=\ldots
\]
where $\adapt(1)$ is equal to the optimal value of the classical robust problem.

However calculating a set of $n+1$ second-stage policies can be computationally heavy. Hence it would be desirable to find smaller bounds for $k$, which lead to a certain approximation guarantee for the optimal value of the two-stage robust problem $\opt(2RO)$. In the following we will show that the results from the previous section can be extended to \eqref{eq:k-adaptability} to provide better bounds on $k$. We will first show that Lemma \ref{lem:bound_opts_optk} and Lemma \ref{lem:multiplicative_bound_opts_optk} also hold for \eqref{eq:k-adaptability}.

Similar to the previous sections, we assume that for a given problem class there exist functions $\underline{p},\bar p: \mathbb N \to \R_+$ such that for each instance $(X,Y(X))$ of the problem class where the number of second-stage variables is $n$ it holds $\underline{p}(n)\le \| y\|_0 \le \bar p(n)$ for all $y\in Y(x)$ for all $x\in X$. We also assume again that $\|\xi\|_\infty \le M_\infty$ and $\|\xi\|_\infty\ge m_\infty>0$ for all $\xi\in U$. Furthermore we assume $d\ge 0$.

\begin{lemma}\label{lem:bound_opts_optk_k-adapt}
Assume $s,k\in [n+1]$ where $s<k$, then it holds
\[
\adapt(s)-\adapt(k)\le M(n) \frac{k-s}{s+1}.
\]
where $M(n):=M_\infty \bar p(n)-m_\infty \underline{p}(n)$.
\end{lemma}
The proof of Lemma \ref{lem:bound_opts_optk_k-adapt} is very similar to the proof of Lemma \ref{lem:bound_opts_optk} and is provided in the Appendix.

We now prove that also Lemma \ref{lem:multiplicative_bound_opts_optk} can be extended to the $k$-adaptable case.
\begin{lemma}\label{lem:multiplicative_bound_opts_optk_k-adapt}
Let $s,k\in [n]$ where $s<k$ and for the first-stage costs it holds $d\ge 0$, then it holds
\[
\adapt(s)\le \left( 1 + \tilde M(n)\frac{k-s}{s+1}\right)\adapt(k)
\]
where $\tilde M(n):=\frac{M_\infty}{m_\infty}\frac{\bar p(n)}{\underline{p}(n)}$. Furthermore if for all instances of dimension $n$ it holds $\underline p(n)=\bar p(n)$, then $\tilde M(n):=\frac{M_\infty}{m_\infty}$ is independent of $n$.
\end{lemma}
\begin{proof}
Similar as in the proof of Lemma \ref{lem:multiplicative_bound_opts_optk} we can estimate
\[
\adapt(k)\ge m_\infty \underline p(n)
\] 
where we used the fact that the first stage costs $d^\top x\ge 0$. We can now follow the proof of Lemma \ref{lem:multiplicative_bound_opts_optk} to prove the result.
\end{proof}

\paragraph*{Approximation guarantees for two-stage robust optimization}
We can now follow the derivations in Section \ref{sec:additive_errors} and \ref{sec:multiplicative_errors} analogously to prove the following bounds on the number of second-stage policies needed to approximate the optimal value of the robust two-stage robust problem up to a certain approximation guarantee.

\begin{lemma}\label{lem:range_l_2St}
Let $k=n-l$, then 
\[
\adapt(k) \le \opt(2RO) + a(n)
\]
for all 
\[
l \in \left[0, \min\left\{n-1,\frac{a(n)n}{M(n)+a(n)}\right\} \right]\cap \mathbb N.
\]
\end{lemma}
From the latter Lemma we can conclude that to derive a solution to \eqref{eq:two-stage_problem} with an additive approximation guarantee of $a(n)$ it is enough to calculate an optimal solution of \eqref{eq:k-adaptability} with \[k=n- \frac{a(n)n}{M(n)+a(n)}\] 
second-stage policies. Furthermore all the bounds $k_0$ provided in Table \ref{tbl:ranges_k} can also be used in this case for \eqref{eq:k-adaptability}. Note that the number of second-stage policies needed in Table \ref{tbl:ranges_k} is significantly smaller than $n$, especially with increasing $n$. This is also shown in Example \ref{ex:facility_location} and \ref{ex:recoverable}.

Analogously as in the latter derivation we can conclude for the multiplicative approximation guarantees.
\begin{lemma}\label{lem:range_q_2St}
Let $k=qn$ and $q\in (0,1]$, then 
\[
\adapt(k) \le (1+a(n))\opt(2RO)
\]
for all 
\[
q \in \left[\frac{\tilde M(n)}{\tilde M(n)+a(n)}, 1\right].
\]
\end{lemma}
From the latter Lemma we can conclude that to derive a solution to \eqref{eq:two-stage_problem} with an multiplicative approximation guarantee of $a(n)$ it is enough to calculate an optimal solution of \eqref{eq:k-adaptability} with \[k=\big\lceil\frac{\tilde M(n)}{\tilde M(n)+a(n)}n\big\rceil\] 
second-stage policies. Furthermore all the bounds $k_0$ provided in Table \ref{tbl:ranges_k_multiplicative} can also be used in this case for \eqref{eq:k-adaptability}. Note that the number of second-stage policies needed in Table \ref{tbl:ranges_k_multiplicative} is significantly smaller than $n$, especially with increasing $n$. This is shown in the following examples.

\begin{example}[Facility Location] \label{ex:facility_location}
Consider the two-stage robust facility location problem (see \cite{hanasusanto2015k}) for a fixed number $s$ of locations and a variable number $r$ of customers. Fixing the number of locations is rarely a restriction in practice since the number of locations can be chosen to be small. We consider the case $s=100$. In the first stage we can open facilities in a subset of the locations which comes with a cost. In the second-stage we have to assign each customer to exactly one of the opened facilities. We assume that the costs of assigning customers to facilities is uncertain. In the classical integer programming formulation the first-stage dimension is $m=s$ and the second-stage dimension is $n=sr$. For each second-stage solution $y$ it holds $\|y\|_0=r$, i.e. $\underline p (n)=\bar p(n)=\frac{1}{s}n$.  

We assume the second-stage assignment costs are given by an uncertainty set $U\subseteq [\bar c- 0.02\bar c, \bar c + 0.02 \bar c]^n$ where $\bar c\in \R^n$ is a given mean vector with entries in $[800,1000]$. Then we have
\[
784 \le \| c\|_\infty \le 1020 \ \ \forall \ c\in U
\]
and hence
\[
\frac{M_\infty}{m_\infty} \le \frac{1020}{784} \le 2.
\]    
We can conclude that $M(n)\le 2.5n$ and $\tilde M(n)\le 2$. Applying the result of Lemma \ref{lem:range_q} we obtain the following number of second-stage policies which have to be used in the $k$-adaptability approach to achieve a additive or multiplicative approximation guarantee of $a(n)$ for \eqref{eq:two-stage_problem}. In brackets we show the fraction of $n$ for $r=1000$ customers.
\begin{table}[h!]
\centering
\begin{tabular}{ccc}
$a(n)$ & additive guarantee & multiplicative guarantee \\
\hline
$1$ & $0.999996 n$&  $\frac{2}{3}n$\\
$\log (n)$ & $0.99998n$& $0.15 n$\\
$\sqrt{n}$ & $0.9987n$ &  $0.007 n$
\end{tabular}
\caption{Number of second-stage policies needed in the $k$-adaptability approach (with $n\ge 1000$) to approximate the recoverable robust optimization problem up to an approximation guarantee of $a(n)$.}
\end{table}
\end{example}

%

\begin{example}[Recoverable Robust Optimization]\label{ex:recoverable}
In the recoverable robust optimization problem (see \cite{liebchen2009concept}) the idea is to find a solution in $X\subseteq \{ 0,1\}^m$ which is worst-case optimal where in the second-stage after the uncertain parameters are revealed we are allowed change at most $\bar p$ entries of the solution and the result has to be another feasible solution in $X$. This can be modeled as a two-stage robust optimization problem as
\[
\min _{x\in X} \ \max_{c\in U} \ \min_{y\in Y(x)} c^\top y
\]
where \[Y(x):\left\{ z^+,z^-\in \{ 0,1\}^m: x+ z^+-z^-\in X, \ \sum_{i=1}^{m} z_i^+ + z_i^- \le \bar p, \ \right\}.\]
Here the dimension of the second-stage is $n=2m$ and $\underline p(n) = \bar p(n)=\bar p$. We consider the same uncertainty set as in Example \ref{ex:facility_location}, i.e. 
\[M(n)=M_\infty\bar p - m_\infty \bar p =  0.04\| \bar c\|_\infty\bar p \le 40\bar p\]
and $\tilde M(n)\le 2$. In the following table we show the values for the number of second-stage policies $k$ we have to use in the $k$-adaptability approach to get an additive (or multiplicative) approximation guarantee of $a(n)$. We show the values for $m=500$ (i.e. $n=1000$) and $\bar p = 10$. Note that the values also holds for dimensions larger than that, however in this case the fraction should be re-calculated to obtain better values.

The result indicates, e.g. if we use the $k$-adaptability approach with $k=60$ second-stage policies, then we obtain a multiplicative $\sqrt{n}$-approximation of the recoverable robust problem.
\begin{table}[h!]
\centering
\begin{tabular}{ccc}
$a(n)$ & additive guarantee & multiplicative guarantee \\
\hline
$1$ & $0.998 n$&  $\frac{2}{3}n$\\
$\log (n)$ & $0.98n$& $0.23n$\\
$\sqrt{n}$ & $0.93n$ &  $0.06n$
\end{tabular}
\caption{Number of second-stage policies needed in the $k$-adaptability approach (with $n\ge 1000$) to approximate the recoverable robust optimization problem up to an approximation guarantee of $a(n)$.}
\end{table}
\end{example}

Furthermore note that Corollary \ref{cor:constant_approx} can directly be extended to $k$-adaptability problems.

\begin{corollary}
The $k$-adaptability problem achieves a constant multiplicative approximation guarantee of
\[
1 + \frac{M_\infty}{m_\infty} \frac{1-q}{q}
\]
for the two-stage robust optimization problem if $k=\lceil qn\rceil$ for a fixed $q\in (0,1)$ and $\underline p(n) = \bar p(n)$.
\end{corollary}
Note that the latter corollary is especially valid for the recoverable robust problem. For budgeted uncertainty sets where each parameter is allowed to deviate from its mean by at most $50\%$ we have $\frac{M_\infty}{m_\infty}=1.5$. If we choose $q=\frac{1}{2}$ the constant approximation factor of Algorithm \ref{alg:approx} is 
\[
1 + \frac{M_\infty}{m_\infty} \frac{1-q}{q} = 2.5.
\]

\subsection{Approximation algorithm for $k$-adaptability problems}
The results in the latter subsection provide useful bounds on the number of second-stage policies $k$ which have to be calculated to obtain a certain approximation guarantee for the exact two-stage robust problem. While this is useful from a theoretical point of view, there is no algorithm known which can achieve these approximation guarantees. Unfortunately using Algorithm \ref{alg:approx} is not possible since it cannot provide first-stage solutions.

In \cite{goerigk2023optimal} the authors derive approximation guarantees for the two-stage robust problem with discrete uncertainty. They show that by choosing an appropriate set of scenarios and solving the two-stage robust problem for this set, the obtained first-stage solution has a certain approximation guarantee. In the following we will adapt the results of \cite{goerigk2023optimal} to derive an approximation algorithm for the $k$-adaptability problem with convex uncertainty sets.

We first adapt and generalize the result from \cite{goerigk2023optimal} to our setting, i.e. for convex uncertainty sets and for the $k$-adaptability problem.

\begin{lemma}\label{lem:goerigk_approx}
Let $\{ \hat \xi^1,\ldots ,\hat \xi^t\}\subset U$ be a set of scenarios and $\alpha\ge 1$ such that
\begin{equation}\label{eq:condition_alpha-dominating-scenarios}
\forall \xi\in U \ \exists i\in [t]: \ \xi\le \alpha \hat \xi^i
\end{equation}
holds. Then any optimal solution $\hat x$ of the problem
\begin{equation}\label{eq:approximation_problem_scenarios}
\begin{aligned}
\min \ & d^\top x + \mu \\
s.t. \quad & \mu \ge (y^i)^\top \hat \xi^i \quad i=1,\ldots ,t \\
& x\in X, \ y^1,\ldots ,y^t\in Y(x)
\end{aligned}
\end{equation}
has a multiplicative approximation guarantee of $\alpha-1$ for the $k$-adaptability problem for all $k\le t $.
\end{lemma}
\begin{proof}
The proof follows similar steps as in \cite{goerigk2023optimal}. Let $\hat x, \hat y^1, \ldots ,\hat y^t$  be an optimal solution of Problem \eqref{eq:approximation_problem_scenarios} and let $x^*, y^{1*}, \ldots, y^{k*}$ be an optimal solution of the $k$-adaptability problem \eqref{eq:k-adaptability}. We bound the $k$-adaptable objective value of solution $\hat x, \hat y^1, \ldots ,\hat y^t$ as follows:
\begin{align*}
d^\top\hat x + \max_{\xi\in U} \min_{i=1,\ldots ,t}\xi^\top \hat y^i &\le d^\top \hat x + \max_{\xi\in \{ \hat \xi^1,\ldots ,\hat \xi^t\}}\min_{i=1,\ldots ,t} \alpha \xi^\top \hat y^i \\
& \le \alpha \left( d^\top \hat x + \max_{\xi\in \{ \hat \xi^1,\ldots ,\hat \xi^t\}}\min_{i=1,\ldots ,t}  \xi^\top \hat y^i\right) \\
& \le \alpha \left( d^\top x^* + \max_{\xi\in \{ \hat \xi^1,\ldots ,\hat \xi^t\}}\min_{i=1,\ldots ,k}  \xi^\top y^{i*}\right) \\
& \le \alpha \left( d^\top x^* + \max_{\xi\in U}\min_{i=1,\ldots ,k}  \xi^\top y^{i*}\right) \\
& = \alpha\cdot \adapt (k)
\end{align*}
where for the first inequality we used Condition \eqref{eq:condition_alpha-dominating-scenarios} and for the second inequality we used $\alpha\ge 1$ and $d^\top \hat x\ge 0$. Note that Problem \eqref{eq:approximation_problem_scenarios} is equivalent to the problem 
\[
\min_{\substack{x\in X \\ y^1,\ldots ,y^t\in Y(x)}} \ d^\top x + \max_{\xi\in \{ \hat \xi^1,\ldots ,\hat \xi^t\}}\min_{i=1,\ldots ,t}  \xi^\top y^i
\]
and since $\hat x, \hat y^1, \ldots ,\hat y^t$ is an optimal solution of the latter problem the third inequality follows since $k\le t$.  The last inequality follows since $\{ \hat \xi^1,\ldots ,\hat \xi^t\}\subset U$.
\end{proof}
The latter lemma shows how to find a solution with an approximation guarantee of $\alpha-1$ for the $k$-adaptability problem. However, deriving the scenario set $\{ \hat \xi^1,\ldots ,\hat \xi^t\}$ which leads to the best approximation guarantee is computationally heavy; see \cite{goerigk2023optimal}. In the following we propose a more efficient heuristic approach in Algorithm \ref{alg:approx_k-adapt} where an arbitrary set of $t$ scenarios in $U$ is used, which leads to a certain approximation guarantee for the $k$-adaptability problem; see Theorem \ref{thm:approx_guarantee_k-adapt}.

\begin{algorithm}\caption{Approximation Algorithm K-adaptability}\label{alg:approx_k-adapt}
	\textbf{Input:} $n\in\N$, $k\in [n]$, $t\ge k$, convex set $U\subset\R^{n}$, $\underline u,\bar u\in \R^n$ s.t. $U\subseteq [\underline u,\bar u]$ 
	\begin{algorithmic}[1]
		\State Select arbitrary $\hat \xi^1, \ldots ,\hat \xi^t\in U$.\label{step:select_scenarios}
		\State Calculate an optimal solution $(\hat x,\hat y^1, \ldots , \hat y^t)$ of Problem \label{step:solve_scenario_problem} \eqref{eq:approximation_problem_scenarios} with scenarios $\hat \xi^1, \ldots ,\hat \xi^t$.
\State Calculate solutions $\bar y^{1},\ldots, \bar y^{k}\in Y(\hat x)$ by Algorithm \ref{alg:approx} for the min-max-min problem \label{step:algorithm1_k-adaptability}
\[
d^\top \hat x + \min_{y^1,\ldots ,y^k\in Y(\hat x)} \max_{\xi\in U} \min_{i=1,\ldots ,k} \xi^\top y^i
\]
		\State \textbf{Return:} $\hat x,\bar y^{1},\ldots ,\bar y^{k}$
	\end{algorithmic}
\end{algorithm}

\begin{remark}[Variant of Algorithm \ref{alg:approx_k-adapt}]
In Remark \ref{rem:variant_alg1} we presented a variant of Algorithm \ref{alg:approx} where for the selection of the $k$ solutions a sparsity constraint is added to the Problem in Step \ref{step:opt_lambda} of Algorithm \ref{alg:approx}. Obviously we directly obtain a variant of Algorithm \ref{alg:approx_k-adapt} by applying Algorithm \ref{alg:approx} (Variant) in Step \ref{step:algorithm1_k-adaptability} of Algorithm \ref{alg:approx_k-adapt}. We will denote this algorithm as Algorithm \ref{alg:approx_k-adapt} (Variant).
\end{remark}

In the following we define the $k$-adaptable objective value for the solution $\hat x,\bar y^1,\ldots ,\bar y^k$ returned by Algorithm \ref{alg:approx_k-adapt} as
\[
\text{approx}_{\text{adapt}}(k):=d^\top \hat x + \max_{\xi\in U}\min_{i=1,\ldots ,k} \ \xi^\top \bar y^i.
\]
We can now show the following theorem.

\begin{theorem}\label{thm:approx_guarantee_k-adapt}
Let $k\in\mathbb N$ with $k\le n$, $U\subset\R_+^{n}$ and $\underline u,\bar u\in \R_+^n$ such that $U\subseteq [\underline u,\bar u]$. Then for the solution $\hat x,\bar y^1,\ldots ,\bar y^k$ returned by Algorithm \ref{alg:approx_k-adapt} it holds
\[
\mathrm{approx}_{\text{adapt}}(k) \le \alpha \left( 1 + \tilde M(n)\frac{n-k}{k+1}\right) \adapt(k) 
\]
where $\alpha=\max_{j=1,\ldots ,n} \frac{\bar u_j}{\underline u_j}$ and $\tilde M(n)$ is defined as in Lemma \ref{lem:multiplicative_bound_opts_optk}.
\end{theorem}
\begin{proof}

We first derive the approximation value $\alpha$ in Lemma \ref{lem:goerigk_approx} for the set of scenarios $\hat \xi^1, \ldots ,\hat \xi^t$ selected in Step \eqref{step:select_scenarios}. Note that for every scenario $\xi\in U$ it holds
\[
\frac{\xi_j}{\hat\xi_j^i} \le \frac{\bar u_j}{\underline u_j} \quad \forall i\in [t], j\in [n]
\]
since $\hat \xi^1, \ldots ,\hat \xi^t\in U$. Hence, the assumption of Lemma \ref{lem:goerigk_approx} is true for
\[
\alpha := \max_{j=1,\ldots ,n} \frac{\bar u_j}{\underline u_j}.
\]
We can now bound the objective value of the solution calculated by Algorithm \ref{alg:approx_k-adapt} as follows:
\begin{align*}
\text{approx}_{\text{adapt}}(k) & =d^\top \hat x + \max_{\xi\in U}\min_{i=1,\ldots ,k} \ \xi^\top \bar y^i \\
& \le d^\top \hat x + \left( 1 + \tilde M(n)\frac{n-k}{k+1}\right)\min_{y^1,\ldots ,y^k\in Y(\hat x)}\max_{\xi\in U}\min_{i=1,\ldots ,k} \ \xi^\top y^i \\
&\le \left( 1 + \tilde M(n)\frac{n-k}{k+1}\right)\left(d^\top \hat x + \min_{y^1,\ldots ,y^k\in Y(\hat x)}\max_{\xi\in U}\min_{i=1,\ldots ,k} \ \xi^\top y^i \right) \\
& \le \alpha\left( 1 + \tilde M(n)\frac{n-k}{k+1}\right) \adapt(k)
\end{align*}
where the first inequality follows from the multiplicative approximation guarantee derived in Corollary \ref{cor:approx_multiplicative} which we obtain since we apply Algorithm \ref{alg:approx} in Step \eqref{step:algorithm1_k-adaptability} to the min-max-min problem with fixed first-stage solution $\hat x$. The second inequality follows since $\left( 1 + \tilde M(n)\frac{n-k}{k+1}\right)\ge 1$ and $d^\top \hat x\ge 0$ and the third inequality follows by applying Lemma \ref{lem:goerigk_approx} with the $\alpha$ value derived above.
\end{proof}

Note that the derived approximation guarantee differs only by the multiplicative factor $\alpha$ from the one shown in Lemma \ref{lem:multiplicative_bound_opts_optk_k-adapt} and $\alpha$ does not depend on the dimension $n$. Hence the analysis shown in Section \ref{sec:kcloseton} can be easily adapted here. However, the approximation guarantee can never be smaller than $\alpha$. The factor $\alpha$ is given as the maximum ratio of the upperbound and the lowerbound value of the uncertainty set over all dimensions. Hence, larger uncertainty sets lead to larger approximation guarantees. Note that if we consider budgeted uncertainty sets where the deviation of the mean-value is at most $50\%$, then $\alpha = \max_{j=1,\ldots ,n}\frac{\bar u_j}{\underline u_j}\approx 1.5$.
 
One limitation of our analysis is that the parameter $\alpha$ is independent of the number of scenarios $t$ which are chosen. This is counter intuitive since generating more scenarios in Step \ref{step:select_scenarios} of Algorithm \ref{alg:approx_k-adapt} can lead to better approximation guarantees.

\subsection{Calculating lower bounds for $k$-adaptability problems}\label{sec:lowerbound_k-adapt}

In the following we show how to efficiently calculate lower bounds for the $k$-adaptability problem which can be used together with the upper bounds calculated in the previous section in classical algorithmic frameworks to calculate optimality gaps.

Using the reformulation of the $k$-adaptability problem
\[
\min_{\substack{x\in X, \ \lambda\ge 0 \\ y = \sum_{i\in [k]}\lambda_iy^i \\ \sum_{i\in [k]} \lambda_i = 1 \\ y^1, \ldots, y^k \in Y(x)}}\max_{\xi\in U}\ d^\top x + \xi^\top y .
\]
it is easy to see that a valid lower bound is given by the case where $k=n$ and in this case the problem is equivalent to
\[
\min_{x\in X, y\in \conv{Y(x)}} \max_{\xi \in U} \ d^\top x + \xi^\top y .
\]
This problem can be lower bounded by 
\[
\min_{(x,y)\in \conv{Z}} \max_{\xi \in U} \ d^\top x + \xi^\top y 
\]
where $Z=\left\{ (x,y): \ x\in X, y\in Y(x) \right\}$. The latter problem is of the type \eqref{eq:minmaxconv} and hence it is equivalent to a min-max-min robust problem over the set $Z$ with $k=n$. This can be efficiently solved again by the standard methods shown in \cite{buchheim2016min}. We will use this lower bound to provide optimality gaps for Algorithm \ref{alg:approx_k-adapt} in Section \ref{sec:computations}.

\section{Computations}\label{sec:computations}
In this section we test Algorithm \ref{alg:approx} (and its variant) on random instances from the min-max-min literature for the shortest path problem and the knapsack problem and compare the objective values to the lower bound obtained by solving M$^3(n)$. 

Additionally we test Algorithm \ref{alg:approx_k-adapt} on a generic two-stage robust problem studied in \cite{ghahtarani2023double} and on a two-stage version of the shortest path problem (see \cite{goerigk2022data}), where in the first stage a set of edges has to be bought, while after the edge costs are revealed in the second-stage a shortest path has to be calculated where only edges bought in the first stage can be used. We apply Algorithm \ref{alg:approx_k-adapt} with $t=n$ scenarios selected in the first step.

All algorithms were implemented in Python 3.10 and all optimization problems were solved by Gurobi 10 \cite{gurobi10} with standard hyperparameter setup. Experiments were executed on a cluster with AMD Genoa 9654 (2x) 96 Cores/Socket 2.4GHz 360W CPU and with 24 x 16GiB 4800MHz, DDR5 RAM. 

\subsection{Instances}
We test the algorithm on the \textbf{minimum knapsack problem (KP)} which is given by
\[
\min_{\substack{a^\top x \ge b\\ \ x\in \{0,1\}^n}} \bar c^\top x
\]
where the instances were generated as in \cite{ChasseinGKP19,goerigk2020min}. For $n\in\{ 50,100\}$ we generate $10$ random knapsack instances where the mean-costs $\bar c_i$ and the weights $a_i$ were drawn from a uniform distribution on $\{1,\ldots ,100\}$ and the knapsack capacity $b$ was set to $35\%$ of the sum of all weights. For each knapsack instance we generated a budgeted uncertainty set 
\[
U=\left\{ c\in\R^n: c_i = \bar c_i + \delta_i \Delta_i, \ \sum_{i\in [n]} \delta_i\le \Gamma, \delta\in [0,1]^n\right\}
\]
where each $\Delta_i$ was drawn uniformly from $\{ 1,\ldots , \bar c_i\}$ with budget parameter $\Gamma\in \{ 2,5,10\}$.

For the \textbf{shortest path problem (SPP)} we use the original instances from \cite{hanasusanto2015k} which were also used in several other publications of the min-max-min literature. We consider random graphs $G=(V,E)$ with $|V|\in \{30,50\}$ nodes corresponding to points in the Euclidean plane with random coordinates in $[0,10]$. For each dimension we randomly select $10$ of the graphs generated in \cite{hanasusanto2015k}. We consider budgeted uncertainty sets described as above where the mean values $\bar c_{ij}$ on edge $(i,j)$ are set to the euclidean distance of node $i$ and $j$ and the deviation values are set to $\Delta_{ij}=\frac{\bar c_{ij}}{2}$. The budget parameter $\Gamma$ was chosen in $\in\{ 2,5,10\}$.

For the two-stage case we consider the \textbf{generic $k$-adaptability problem (GP)} studied in \cite{ghahtarani2023double} which is defined as
\begin{align*}
\min & \max_{\xi\in U} \min_{i=1,\ldots ,k} d^\top x + \xi^\top y^i\\
s.t. \quad &\sum_{i=1}^{m} x_i=10 \\
& f^\top y^i - a^\top x\ge 0 \quad i=1,\ldots ,k\\
& x\in \{ 0,1\}^m,y^1,\ldots ,y^k\in \{ 0,1\}^n
\end{align*}
where $x$ are the first-stage decision variables and $y^i$ the second-stage decision variables and $n=m=50$. The parameters $d_i$, $a_i$ and $f_i$ are uniformly and randomly generated in $[8,12]$, $[50,100]$ and $[80,90]$, respectively. We generate budgeted uncertainty sets as above with mean values $\bar \xi_i$ uniformly and randomly generated from $[8,12]$ and deviations values $\Delta_{i}=\frac{\bar \xi_{i}}{4}$. The budget parameter $\Gamma$ was chosen in $\in\{ 2,5,10\}$.

Finally we consider a \textbf{two-stage network construction variant of the shortest path problem (2SP)}. We consider the same instances as for the shortest path problem but study the $k$-adaptability problem
\begin{align*}
\min & \max_{\xi\in U} \min_{i=1,\ldots ,k} d^\top x + \xi^\top y^i \\
s.t. \quad & x\in \{ 0,1\}^E \\
& y^1,\ldots ,y^k\in Y_{\text{SP}}(x)
\end{align*}
where $Y_{\text{SP}}(x)$ is the set of paths in $G$ (encoded by their $0$-$1$ incidence vectors) where $y^i_{e}\le x_e$ for all $i\in [k]$. Hence, the task is to buy a set of edges in the first-stage (indicated by the decisions $x$) such that for every scenario $\xi\in U$ a good path $y^i$ exists which only uses edges bought in the first stage. This problem was already studied e.g. in \cite{goerigk2022data}. The first-stage costs are set to $d=\frac{1}{2}\min_{i=1,\ldots ,n}\bar \xi_i\1$ where $\bar \xi$ is the mean-cost vector of the budgeted uncertainty set. The latter choice is motivated by the observation that if the costs are set significantly smaller, the optimal solution is to buy all edges of the graph. On the other hand if the value is set significantly larger the optimal solution is to buy just one path in the graph, so in the second-stage there is no option to adapt to the scenarios.

\subsection{Computational Results}
In the following subsections we show the achieved optimality gaps and runtimes of Algorithm \ref{alg:approx} (and its variant) and Algorithm \ref{alg:approx_k-adapt} (and its variant) for the problem instances introduced in the previous subsection and for different values of $k$. All values are given as the average over $10$ random instances.

To obtain the optimality gap, we first run Algorithm \ref{alg:approx} (or Algorithm \ref{alg:approx_k-adapt}) and record the objective value $\text{obj}$ of the returned solution. Then we calculate a lowerbound for the KP and SPP instances by calculating the optimal value of the corresponding min-max-min problem for $k=n$ by applying the iterative method presented in Section \ref{sec:lowerbound_k-adapt} to the one-stage problems. For the two-stage problems GP and 2SP we apply the same algorithm presented in \ref{sec:lowerbound_k-adapt} to obtain a lower bound. The optimality gap is defined as
\[
\text{gap}:=\frac{\text{obj}-\text{lowerbound}}{\text{lowerbound}}\cdot 100.
\]

The runtimes are given in seconds. We set a timelimit of 1800 seconds for the problem to be solved in Step \ref{step:opt_lambda} of Algorithm \ref{alg:approx} (Variant) since in our experiments it was the only subproblem which sometimes lead to huge runtimes, due to the sparsity constraint. In case the timelimit is reached we take the best known solution provided by Gurobi.

\subsubsection{The Knapsack Problem}
In the following we show the results of Algorithm \ref{alg:approx} and Algorithm \ref{alg:approx} (Variant) on the KP instances described above. In Figure \ref{fig:kp_gaps} the optimality gaps are shown in $\%$ for different values of $k$ for the KP. The results show exactly the behaviour we analyzed in Section \ref{sec:kcloseton}, namely the optimality gap decreases for increasing $k$ and returns the optimal solution already for $k\ge 10$ for both instance sizes. For $k=4$ the optimality gap is at most $1.5\%$. The results show that the optimality gaps increase for larger uncertainty sets (i.e. larger $\Gamma$) and for a larger dimension of the problem. Furthermore Algorithm \ref{alg:approx} (Variant) performs significantly better with only small increases in computation times as shown in Figure \ref{fig:kp_times}. Both algorithms are able to achieve small optimality gaps in less than $1$ second. 

\begin{figure}[h!]
\centering
\includegraphics[width=0.48\textwidth]{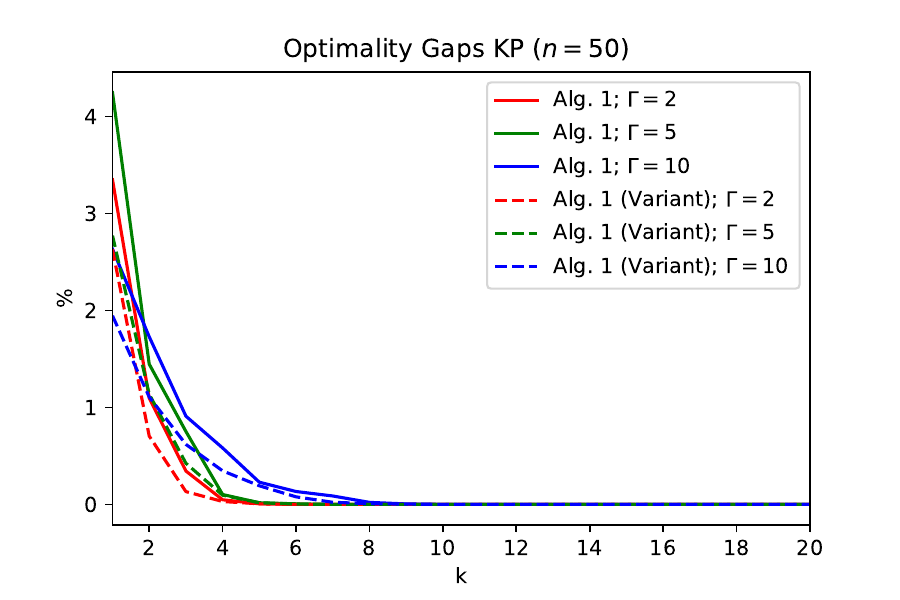}
\includegraphics[width=0.48\textwidth]{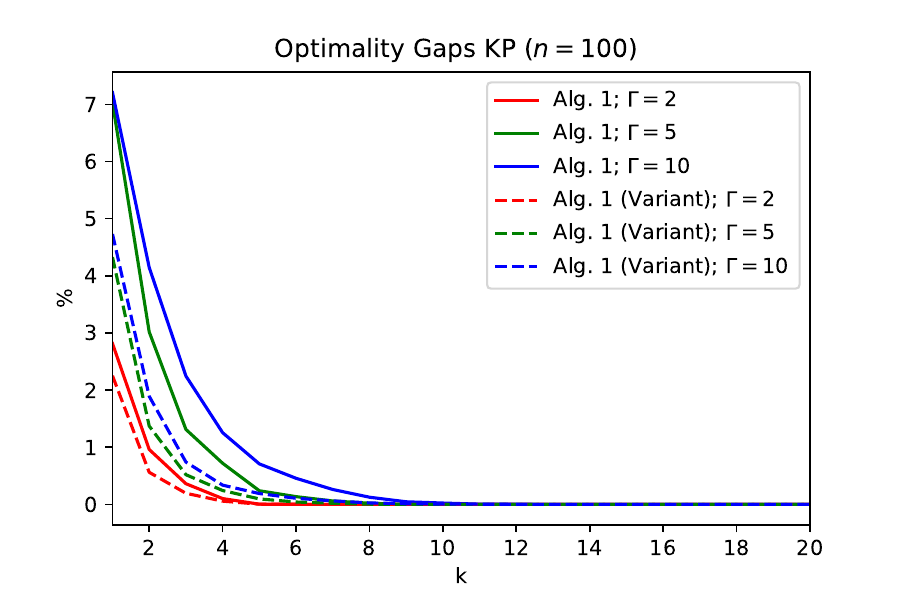}
\caption{Average optimality gaps for KP instances with $n=50$ (left) and $n=100$ (right) for Algorithm \ref{alg:approx} and Algorithm \ref{alg:approx} (Variant) for different values of $\Gamma$.}
\label{fig:kp_gaps}
\end{figure}

\begin{figure}[h!]
\centering
\includegraphics[width=0.48\textwidth]{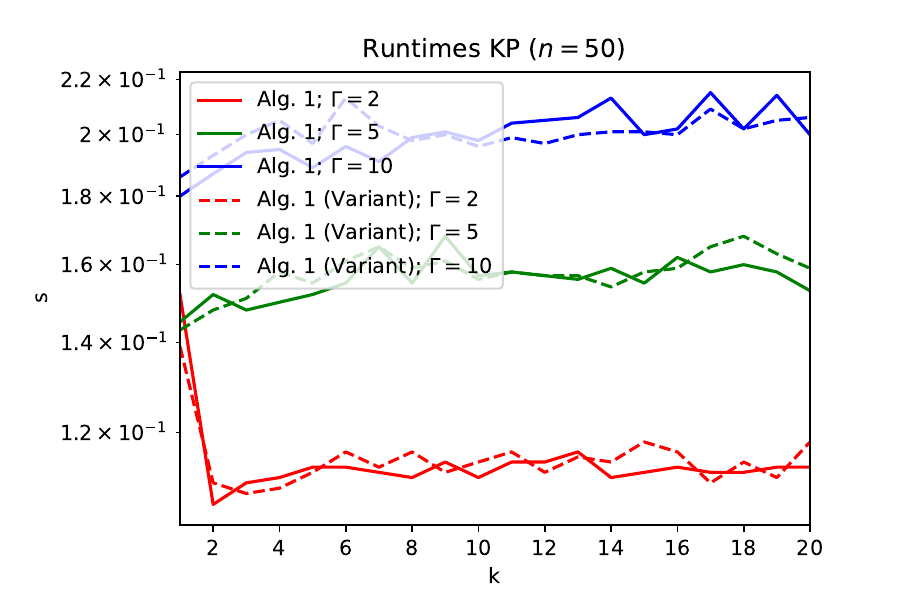}
\includegraphics[width=0.48\textwidth]{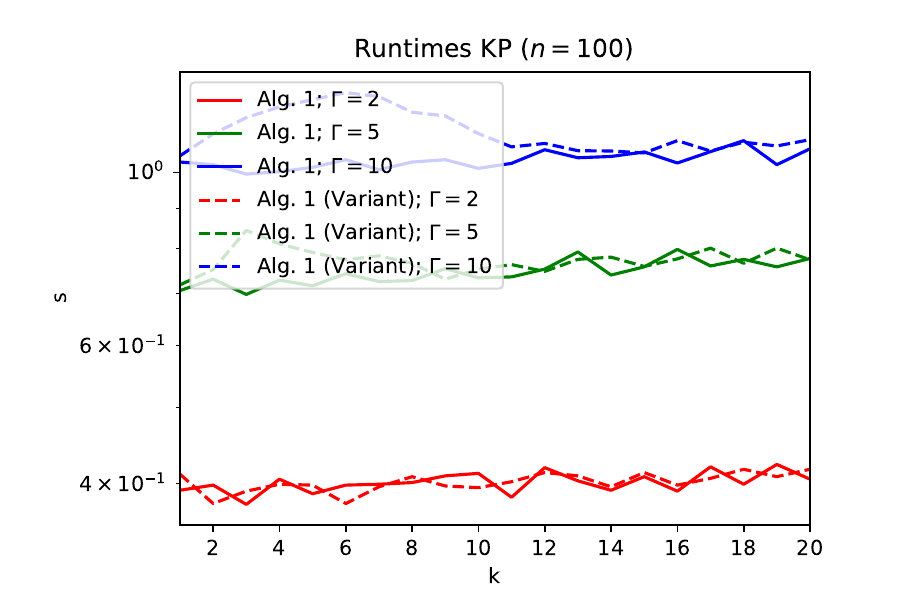}
\caption{Average runtimes for KP instances with $n=50$ (left) and $n=100$ (right) for Algorithm \ref{alg:approx} and Algorithm \ref{alg:approx} (Variant) for different values of $\Gamma$.}
\label{fig:kp_times}
\end{figure}

\subsubsection{Shortest Path Problem}
In the following we show the results of Algorithm \ref{alg:approx} and Algorithm \ref{alg:approx} (Variant) on the SPP instances described above. In Figure \ref{fig:spp_gaps} the optimality gaps are shown in $\%$ for different values of $k$ for the SPP. Similar to the previous section, the results show exactly the behaviour we analyzed in Section \ref{sec:kcloseton}, namely the optimality gap decreases for increasing $k$ and returns the optimal solution already for $k\approx 20$ for both instance sizes. Compared to KP the optimality gaps are larger for small $k$. The results show that the optimality gaps increase for larger uncertainty sets (i.e. larger $\Gamma$) and for a larger dimension of the problem. Furthermore Algorithm \ref{alg:approx} (Variant) performs significantly better. However this comes with larger computation times as shown in Figure \ref{fig:spp_times}. Interestingly the instances with midsize $k$ are the hardest instances for Algorithm \ref{alg:approx} (Variant) while for larger $k$ the runtime decreases again. Possibly this is because the sparsity constraint is redundant for larger $k$. Algorithm \ref{alg:approx} (Variant) is able to achieve optimality gaps below $10\%$ for $k=4$ and $\Gamma=5$ in less than $10$ seconds.

\begin{figure}[h!]
\centering
\includegraphics[width=0.48\textwidth]{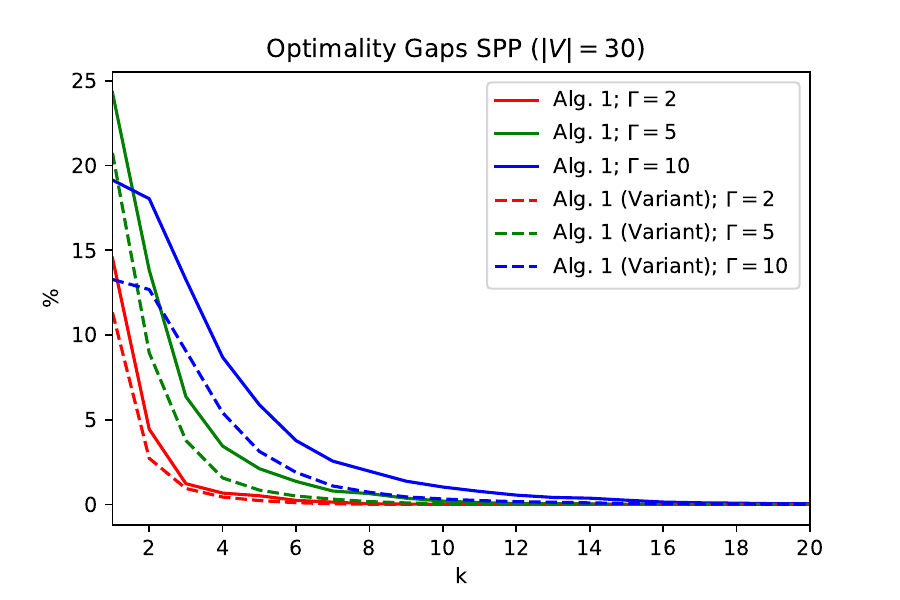}
\includegraphics[width=0.48\textwidth]{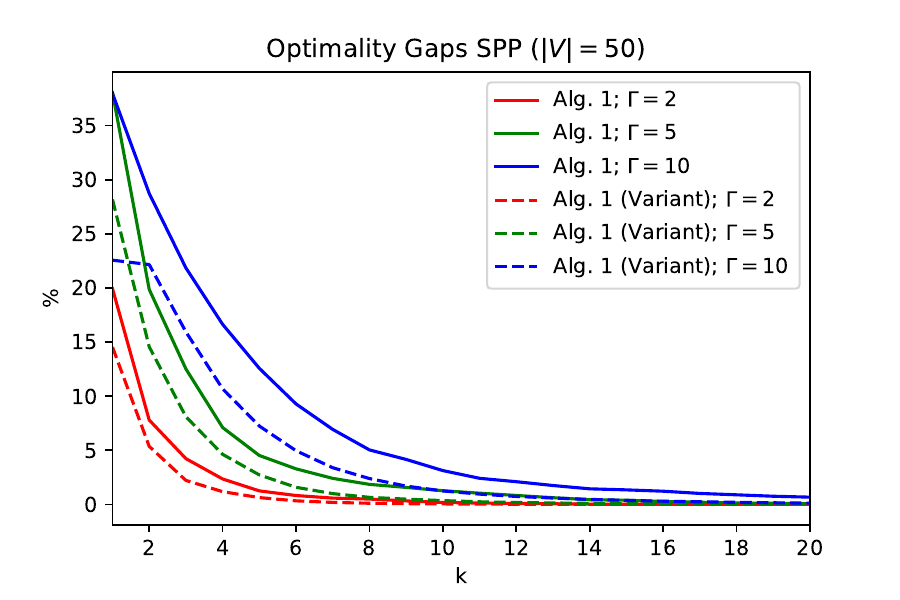}
\caption{Average optimality gaps for SPP instances with $30$ nodes (left) and $50$ nodes (right) for Algorithm \ref{alg:approx} and Algorithm \ref{alg:approx} (Variant) for different values of $\Gamma$.}
\label{fig:spp_gaps}
\end{figure}

\begin{figure}[h!]
\centering
\includegraphics[width=0.48\textwidth]{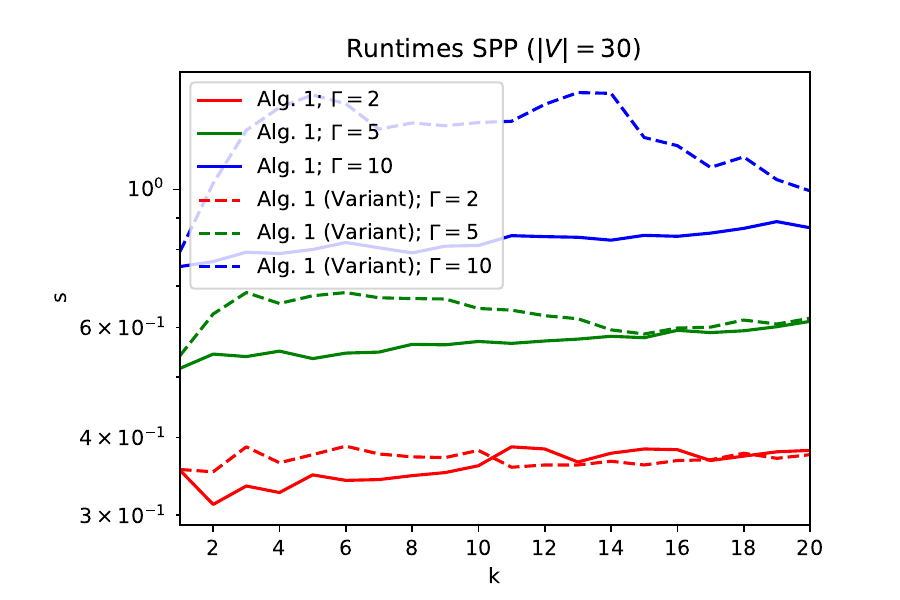}
\includegraphics[width=0.48\textwidth]{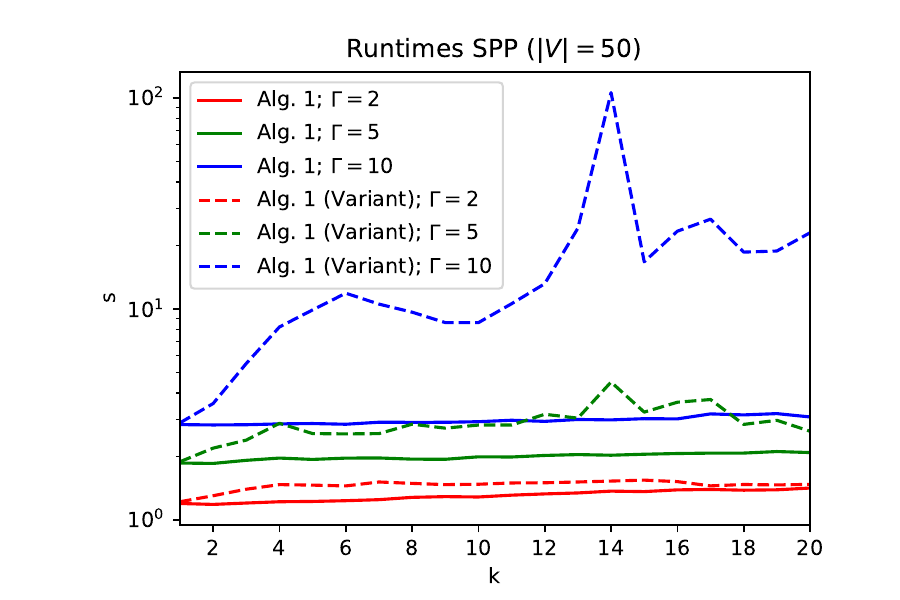}
\caption{Average runtimes for SPP instances with $30$ nodes (left) and $50$ nodes (right) for Algorithm \ref{alg:approx} and Algorithm \ref{alg:approx} (Variant) for different values of $\Gamma$.}
\label{fig:spp_times}
\end{figure}

\subsubsection{A Generative $K$-adaptability Problem}
In the following we show the results of Algorithm \ref{alg:approx_k-adapt} and Algorithm \ref{alg:approx_k-adapt} (Variant) on the GP instances described above. In Figure \ref{fig:gp_gaps} the optimality gaps are shown in $\%$ for different values of $k$. Similar to the previous section, the results show exactly the behaviour we analyzed in Section \ref{sec:kcloseton}, namely the optimality gap decreases for increasing $k$. However since the first-stage solution calculated by Algorithm \ref{alg:approx_k-adapt} is only an approximation, the optimality gap does not necessarily converge to zero for growing $k$ as it was the case in the previous sections. The optimality gaps for small $k$ are not larger than $5\%$. The results show that the optimality gaps increase for larger uncertainty sets (i.e. larger $\Gamma$) and for a larger dimension of the problem. Furthermore Algorithm \ref{alg:approx_k-adapt} (Variant) performs significantly better. This comes with a small increase in computation times as shown in Figure \ref{fig:gp_times}. Similar to SPP the instances with midsize $k$ are the hardest instances for Algorithm \ref{alg:approx} (Variant) while for larger $k$ the runtime decreases again. Possibly this is because the sparsity constraint is redundant for larger $k$. Algorithm \ref{alg:approx_k-adapt} (Variant) is able to achieve optimality gaps below $2\%$ for $k=4$ in less than $1000$ seconds.

\begin{figure}[h!]
\centering
\includegraphics[width=0.48\textwidth]{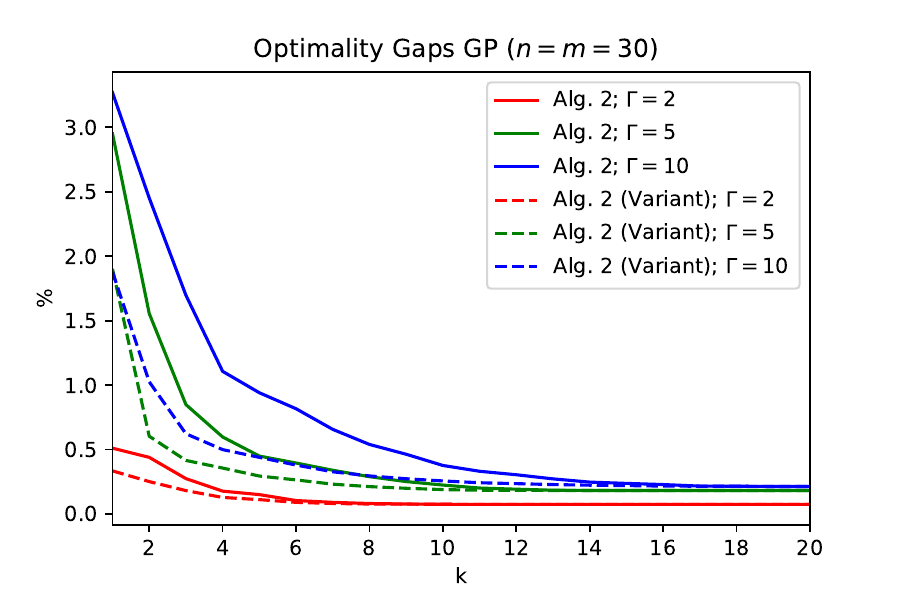}
\includegraphics[width=0.48\textwidth]{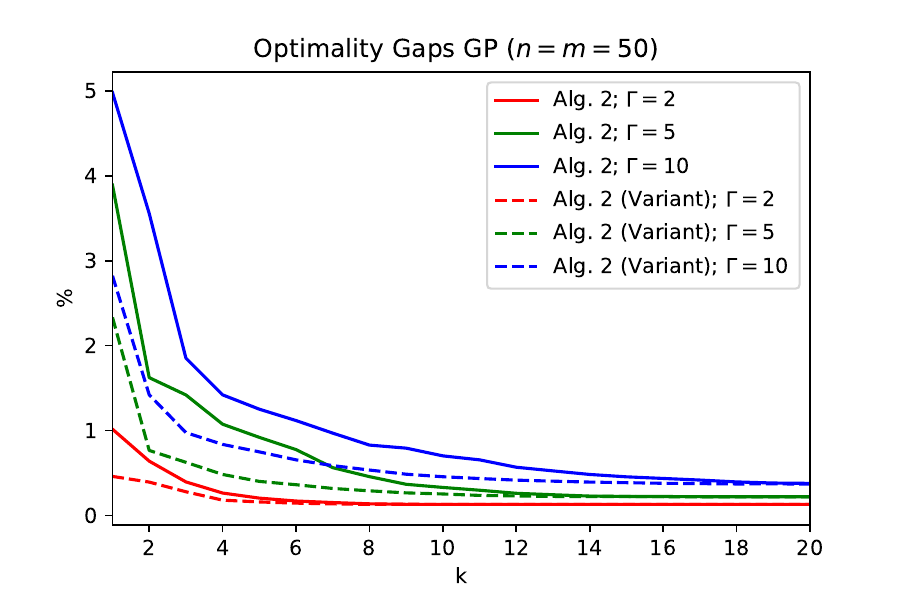}
\caption{Average optimality gaps for GP instances with $n=m=30$ (left) and $n=m=50$ (right) for Algorithm \ref{alg:approx_k-adapt} and Algorithm \ref{alg:approx_k-adapt} (Variant) for different values of $\Gamma$.}
\label{fig:gp_gaps}
\end{figure}

\begin{figure}[h!]
\centering
\includegraphics[width=0.48\textwidth]{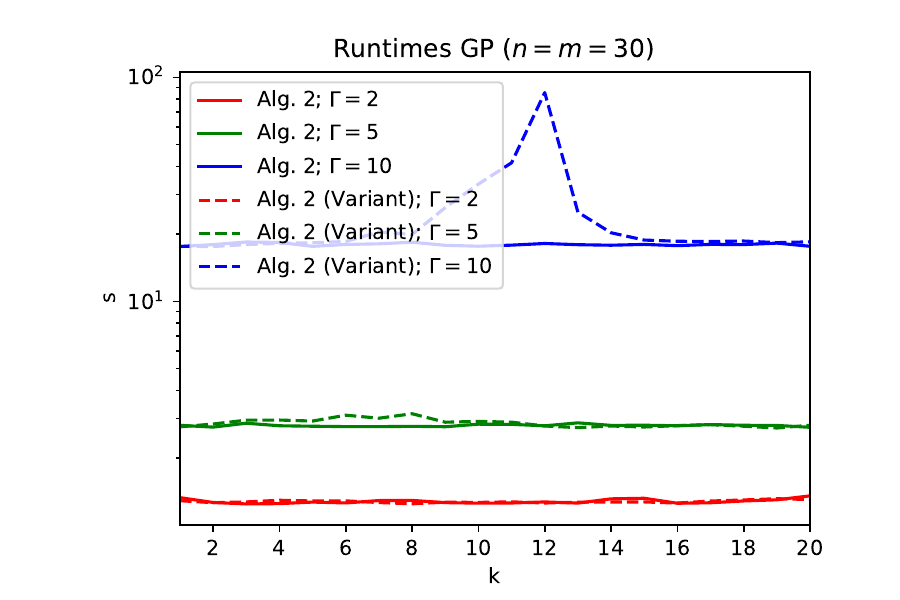}
\includegraphics[width=0.48\textwidth]{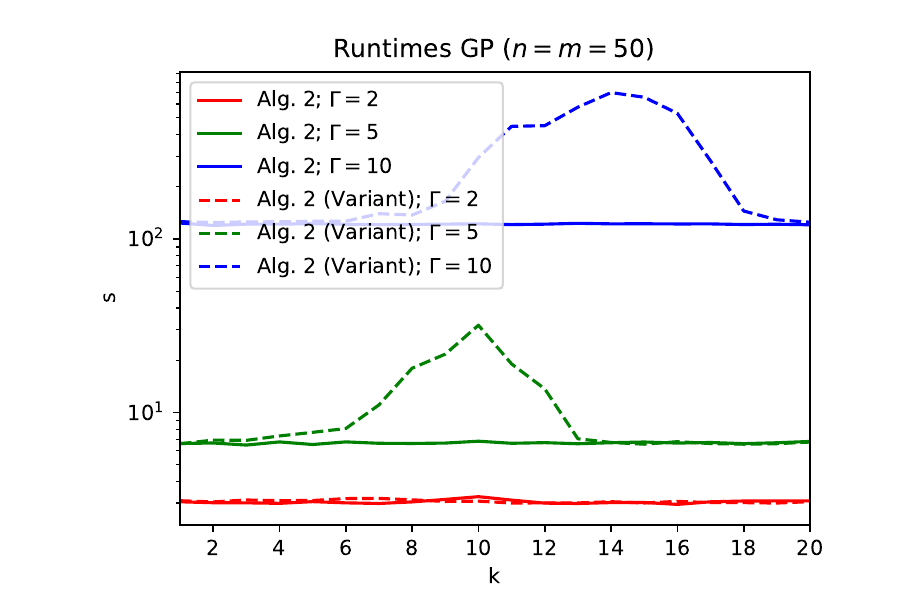}
\caption{Average runtimes for GP instances with $n=m=30$ (left) and $n=m=50$ (right) for Algorithm \ref{alg:approx_k-adapt} and Algorithm \ref{alg:approx_k-adapt} (Variant) for different values of $\Gamma$.}
\label{fig:gp_times}
\end{figure}

\subsubsection{Network Construction Shortest Path Problem}
In the following we show the results of Algorithm \ref{alg:approx_k-adapt} and Algorithm \ref{alg:approx_k-adapt} (Variant) on the 2SP instances described above. In Figure \ref{fig:2sp_gaps} the optimality gaps are shown in $\%$ for different values of $k$. The optimality gap decreases for increasing $k$. However since the first-stage solution calculated by Algorithm \ref{alg:approx_k-adapt} is only an approximation, the optimality gap does not necessarily converge to zero for growing $k$ as it was the case for the min-max-min instances. The approximate first-stage solution is significantly worse compared to GP and the optimality gaps remain large. Here better problem-specific approximation algorithms for finding good two-stage solutions have to be developed. The results show that the optimality gaps increase for larger uncertainty sets (i.e. larger $\Gamma$) and for a larger dimension of the problem. Both, Algorithm \ref{alg:approx_k-adapt} and  Algorithm \ref{alg:approx_k-adapt} (Variant) perform similarly in terms of optimality gap and runtimes. Interestingly, the runtimes for instances with $\Gamma=10$ are smaller for Algorithm \ref{alg:approx_k-adapt} (Variant). In total the runtime is always smaller than $600$ seconds.

\begin{figure}[h!]
\centering
\includegraphics[width=0.48\textwidth]{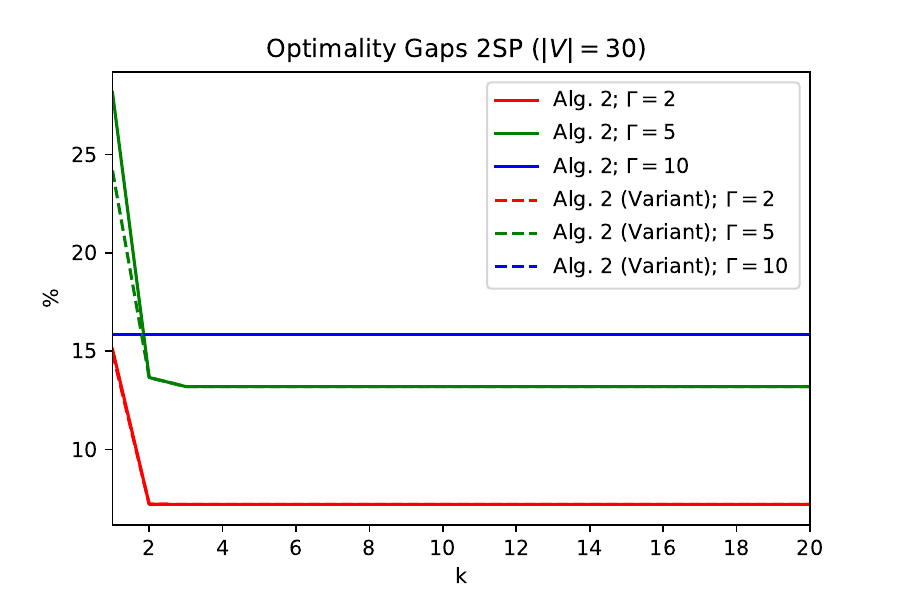}
\includegraphics[width=0.48\textwidth]{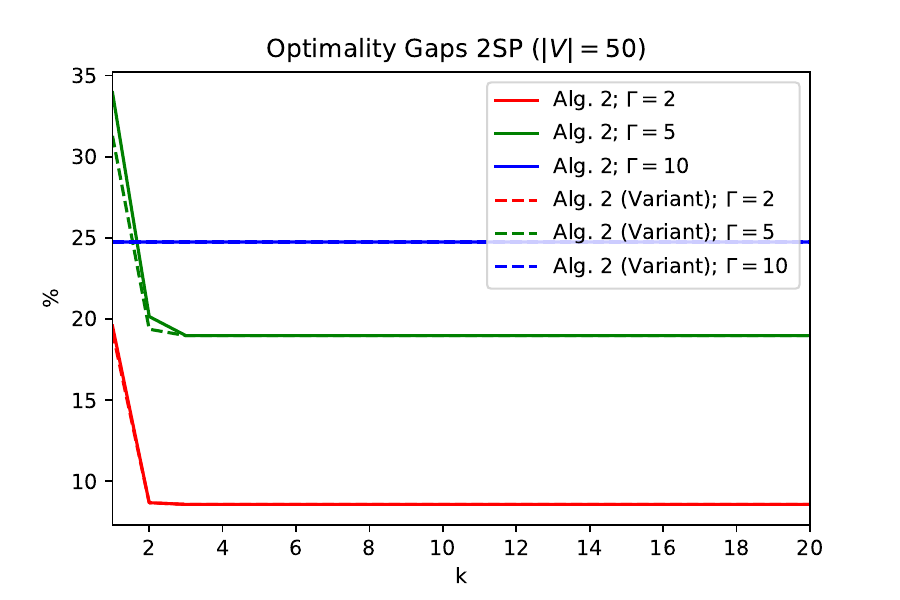}
\caption{Average optimality gaps for 2SP instances with $30$ nodes (left) and $50$ nodes (right) for Algorithm \ref{alg:approx_k-adapt} and Algorithm \ref{alg:approx_k-adapt} (Variant) for different values of $\Gamma$.}
\label{fig:2sp_gaps}
\end{figure}

\begin{figure}[h!]
\centering
\includegraphics[width=0.48\textwidth]{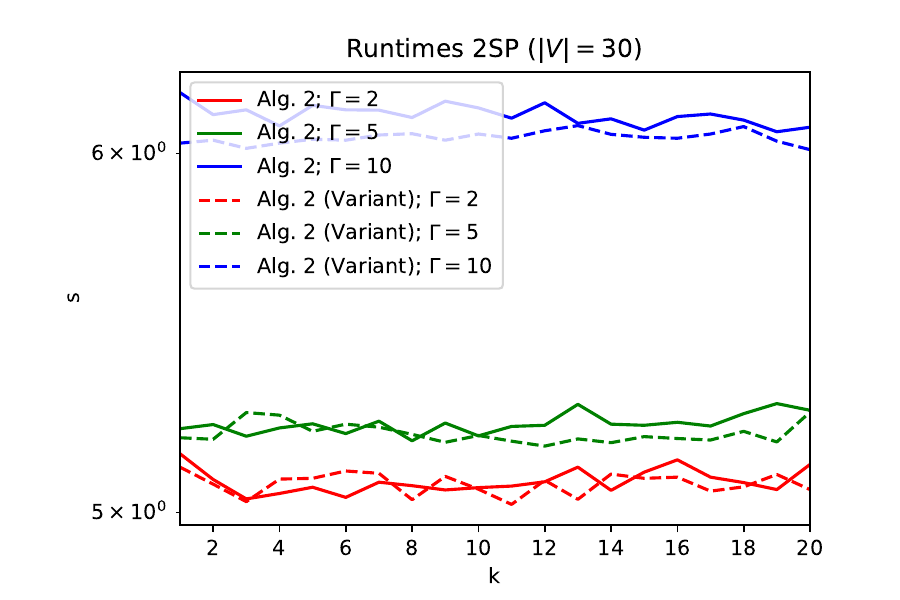}
\includegraphics[width=0.48\textwidth]{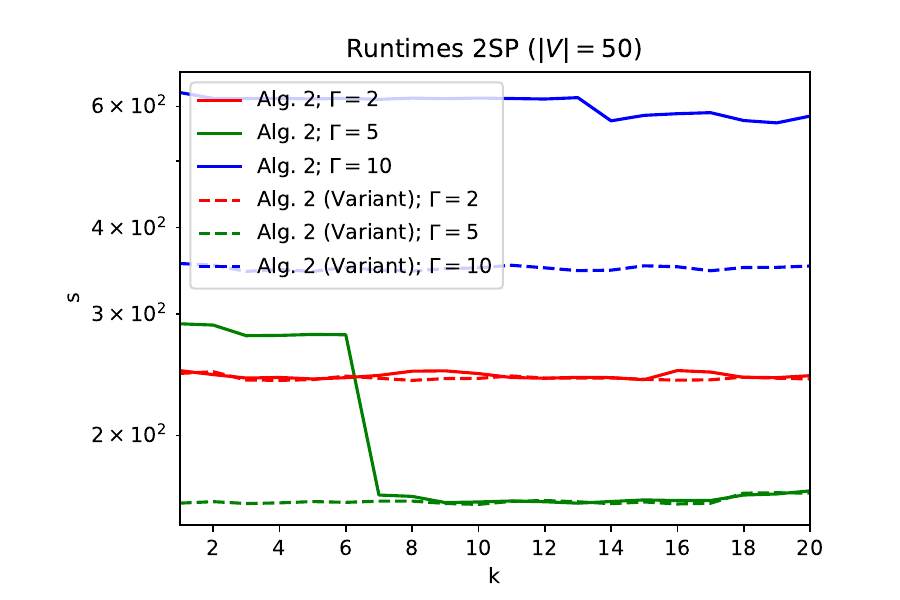}
\caption{Average runtimes for 2SP instances with $30$ nodes (left) and $50$ nodes (right) for Algorithm \ref{alg:approx_k-adapt} and Algorithm \ref{alg:approx_k-adapt} (Variant) for different values of $\Gamma$.}
\label{fig:2sp_times}
\end{figure}

\section{Conclusion}\label{sec:conclusion}
In this work we study the min-max-min robust problem for binary optimization problems with uncertain costs. We focus on the case when the number of calculated solutions $k$ is of intermediate size, i.e. either it is smaller but close to $n$ or a fraction of $n$. We present an algorithm with provable additive and multiplicative approximation guarantees (involving problem-specific parameters) for both cases. The derived approximation guarantees depend on the size of the uncertainty set and the dimension of the problem and can be large for small $k$, but converge quickly to zero with increasing $k$. For special cases the guarantees provide an constant approximation guarantee (independent of $n$). Furthermore we calculate the range of parameter $k$ for which a certain guarantee is achieved. We show that surprisingly if $k$ is close to $n$, the min-max-min problem remains theoretically tractable and can be solved exactly and approximately in polynomial time for a large set of combinatorial problems if some problem parameters are fixed.

We extend the previous results to derive an approximation algorithm for $k$-adaptability problems and to derive lower bounds on the number of second-stage policies $k$ which have to be used to achieve a certain approximation guarantee for the exact two-stage robust problem. We show that these bounds can also be used to approximate the recoverable robust problem. 

Finally, the theoretical insights are confirmed by computations on min-max-min and k-adaptability instances from the literature. The experiments show that the presented algorithm is able to calculate solutions with small optimality gap in seconds, where the optimality gaps quickly go to zero with increasing $k$.

The results of this paper indicate that solving the min-max-min problem (and to some extend the $k$-adaptability problem) gets easier if $k$ gets larger. This is in contrast to the results usually presented in the min-max-min and $k$-adaptability literature, where the number of solved instances (during a given timelimit) is often decreasing in $k$. However, if the approximation algorithms derived in this work (and the presented lower bounds) would be used to calculate solutions with good optimality guarantee, most of the known exact algorithms will scale much better for growing $k$ due to the presented behaviour of the approximation bounds. Hence, we strongly recommend to use the presented approximation algorithms and lower bound calculations. 

While the approximation quality in dependence of $k$ for objective uncertainty is now well understood, for future work it would be interesting to derive similar results for the case that uncertain parameters are also allowed to appear in the constraints. However, the analysis presented in this work cannot be directly extended to this case and a completely new approach has to be developed.

\bibliographystyle{abbrv}      
\bibliography{references} 

\section*{Appendix}
\subsection*{Proof of Lemma \ref{lem:bound_opts_optk_k-adapt}}
\begin{proof}
Let $x^*(k)\in X$, $y^*(k)=\sum_{i\in [k]}\lambda_i y^i$ be an optimal solution of Problem \eqref{eq:k-adaptability} with parameter $k$. We may assume without loss of generality that $\lambda_1\ge \ldots \ge \lambda_k$. Define a solution $x(s)$, $y(s)$ of the $s$-adaptable problem by setting $x^*(s)=x^*(k)$ and 
\begin{equation}\label{eq:define_s_solution}
y(s):=\sum_{i\in [s-1]} \lambda_i y^i + \left(\sum_{i=s}^{k}\lambda_i\right) y^s.
\end{equation}
The latter solution is clearly feasible for the $s$-adaptable problem and hence 
 \[\adapt(s)\le \max_{\xi\in U}d^\top x(s) + \xi^\top y(s).\]
Furthermore let $\xi^*(s)\in \argmax_{\xi\in U}d^\top x(s) + \xi^\top y(s)$. It follows
\begin{align*}
\adapt(s)-\adapt(k) & \le \max_{\xi\in U}d^\top x(s) + \xi^\top y(s) - (\max_{\xi\in U}d^\top x^*(k) + \xi^\top y^*(k)) \\
& \le \xi^*(s)^\top \left( y(s)-y^*(k)\right) \\
& = \sum_{i=s+1}^{k}\lambda_i \xi^*(s)^\top \left( y^s-y^i\right)\\
& \le M(n) \left( \sum_{i=s+1}^{k}\lambda_i \right)
\end{align*}
where the second inequality holds since $\xi^*(s)$ is a subgradient of the function $g(y)=\max_{\xi\in U}d^\top x + \xi^\top y$ in  $y(s)$ and since $x^*(k)=x(s)$. For the first equality we used the definition of $y(s)$ and $y^*(k)$ and for the last inequality we used the assumption $\underline{p}(n)\le \|y\|_0 \le \bar p(n)$ and $m_\infty \le \xi^*(s)_j\le M_\infty$ for all $j\in [n]$. Due to the sorting $\lambda_1\ge \ldots \ge \lambda_k$ and since the sum over all $\lambda_i$ is one, we have $\lambda_i\le \frac{1}{i}$ and hence
\begin{align*}
M(n) \left( \sum_{i=s+1}^{k}\lambda_i \right) &\le M(n) \left( \sum_{i=s+1}^{k}\frac{1}{i} \right)\\
& \le M(n) \frac{k-s}{s+1}
\end{align*}
which proves the result.
\end{proof}

\end{document}